\documentclass[11 pt]{article}
\usepackage{graphicx}
\usepackage{hw-shortcuts}
\usepackage{xcolor}
\usepackage{amsfonts, amsmath, amssymb, amsthm}
\usepackage{quiver}
\usepackage{enumerate}
\usepackage{mathpazo}
\usepackage{authblk}
\usepackage{geometry}
\usepackage{setspace}
\usepackage{thmtools}

\usepackage[pdftex,pagebackref=false]{hyperref} 
\hypersetup{
    plainpages=false,       
    unicode=false,          
    pdftoolbar=true,        
    pdfmenubar=true,        
    pdffitwindow=false,     
    pdfstartview={FitH},    
    pdfnewwindow=true,      
    colorlinks=true,        
    linkcolor=blue,         
    citecolor=green,        
    filecolor=magenta,      
    urlcolor=cyan           
}

\newtheorem{theorem}{Theorem}[section]

\newtheorem{lemma}[theorem]{Lemma}
\newtheorem{proposition}[theorem]{Proposition}

\declaretheorem[name=Definition,style=definition,qed=$\blacktriangle$,numberlike=theorem]{definition}
\declaretheorem[name=Example,style=definition,qed=$\blacktriangle$,numberlike=theorem]{example}
\declaretheorem[name=Remark,style=definition,qed=$\blacktriangle$,numberlike=theorem]{rmk}

\numberwithin{equation}{section}

\title{Some results on calibrated submanifolds in Euclidean space of cohomogeneity one and two}
\author{Faisal Romshoo \\ \vspace{-0.25 cm}
\emph{Department of Pure Mathematics, University of Waterloo} \\
\tt{fromshoo@uwaterloo.ca}}
\date{August 1, 2025}

\begin{document}

\maketitle
\begin{abstract}
    We construct calibrated submanifolds in Euclidean space invariant under the action of a Lie group $G$. We first demonstrate the method used in this paper by reproducing the results about special Lagrangians in Harvey-Lawson \cite{harveylawson}. We then show explicitly that an associative submanifold in $\R^7$ invariant under the action of a maximal torus $\bT^2 \subset \Gt$ has to be a special Lagrangian submanifold in $\C^3$. Similarly, we also show that a Cayley submanifold in $\R^8$ invariant under the action of a maximal torus $\bT^3 \subset \Spin(7)$ has to be a special Lagrangian submanifold in $\C^4$. We construct coassociative submanifolds in $\R^7$ invariant under the action of $\Sp(1)\subset \bH$ with a more general ansatz than the one in \cite{harveylawson} but we recover exactly the $\Sp(1)$-invariant coassociatives in \cite{harveylawson}, giving us a rigidity result. Finally, we construct cohomogeneity two examples of coassociative submanifolds in $\R^7$ which are invariant under the action of a maximal torus $\bT^2 \subset \Gt$.
\end{abstract}

\section{Introduction}
While the underlying principle of \textit{calibrated submanifolds} was first identified by Berger (see \cite[Section 6.5]{MR3497381} for details), the concept was formally introduced and developed by Harvey and Lawson \cite{harveylawson} in their seminal 1982 paper. They are a special class of minimal (vanishing mean curvature) submanifolds of a Riemannian manifold $M$ defined by a particular type of closed form on $M$ known as a \textit{calibration}. They are characterized by a first order non-linear PDE and are locally volume-minimizing in their homology class. In Section \ref{calibrations}, we describe calibrated submanifolds in detail.  In geometry, it is a general
principle that objects with symmetries are more interesting
to study. In this paper, we consider a method of constructing calibrated submanifolds which are invariant under the action of a Lie group $G$. \\

A  list of examples of constructions of calibrated submanifolds with symmetries is as follows:
\begin{itemize}
    \item $\bT^{n-1}$-invariant special Lagrangians (SLs) \cite{goldstein}.
    \item $\bT^{n-1}$-invariant SLs, $\SO(n)$-invariant SLs, $\Sp(1)$-invariant coassociatives \cite{harveylawson}.
    \item $U(1)$-invariant SL cones in $\C^3$ \cite{haskins2000speciallagrangiancones}.
    \item  $\SO(3)$-invariant SL $3$-folds, $\bT^2$-invariant $3$-folds \cite{minoo}.
    \item $\mathrm{U}(1)$-invariant SL $3$-folds, $\mathrm{U}(1)^{m-1}$-invariant SL cones in $\C^m$ \cite{joyce2001EESL}, \cite{joyce2001EQ}, \cite{joyce2001sym}, \cite{joyce2002u1}, \cite{joyce2004i},\cite{joyce2004ii},\cite{joyce2004iii}.
    \item $\mathrm{U(1)}$-invariant associative cones, $\SU(2)$-invariant Cayleys \cite{lotay2006constructingassociative3foldsevolution}, \cite{Lotay_2006}.
\end{itemize}
For a list of the other methods used to construct calibrated submanifolds see \cite[Section 5]{lotay2018calibratedsubmanifolds}. We explain all the steps of the method we use throughout this paper in Section \ref{method}. We use this method to construct different types of calibrated submanifolds in the rest of the sections. We consider special Lagrangians in Sections \ref{maxtorusSL} and \ref{soninvariant},  associatives in $\R^7$ in Section \ref{t2invariantassoc}, coassociatives in $\R^7$ in Sections \ref{sp1invariant} and \ref{t2invariantcoass}, and Cayleys in $\R^8$ in Section \ref{cayley}.\\

We reproduce and review some results about SL $n$-folds of Harvey-Lawson \cite{harveylawson} in Sections \ref{maxtorusSL} and \ref{soninvariant}  to fix notation and prepare the reader for our new constructions in the later sections.  We also construct $\Sp(1)$-invariant coassociatives in Section \ref{sp1invariant}, where we begin with a more general ansatz but recover exactly the $\Sp(1)$-invariant coassociatives in \cite{harveylawson} (Theorem \ref{sp1rigidity}).  \\

Motivated by the $\bT^{n-1}$-invariant SL example in \cite{harveylawson}, we then look at calibrated submanifolds invariant under the action of a maximal torus. We show explicitly in Section \ref{t2invariantassoc} that an associative submanifold in $\R^7$ invariant under the action of a maximal torus $\bT^2 \subset \Gt$ has to be a special Lagrangian submanifold in $\C^3$ (Theorem \ref{assocslthm}). In a similar vein, in Section \ref{cayley}, we show explicitly that a Cayley submanifold in $\R^8$ invariant under the action of a maximal torus $\bT^3 \subset \Spin(7)$ has to be a special Lagrangian submanifold in $\C^4$ (Theorem \ref{cayleysl}).\\

We construct examples of coassociative submanifolds in $\R^7$ which are invariant under the action of a maximal torus $\bT^2 \subset \Gt$ in Section \ref{t2invariantcoass}, which is new. These examples have cohomogeneity two. That is, the orbits of the group are of codimension two.\\ 

\textbf{Acknowledgements}. The author would like to express gratitude to his advisor, Spiro Karigiannis, for his immense help and guidance throughout the development of this paper. The author would also like to thank Jason Lotay for his helpful comments regarding the results in Sections \ref{t2invariantassoc} and \ref{cayley}. The author would also like to thank the anonymous referees for their careful reading and valuable suggestions that improved this paper.

\section{Calibrated Geometry}
\label{calibrations}

 We start by defining \textit{calibrations} and \textit{calibrated submanifolds} on a Riemannian manifold $(M,g)$. More about calibrated geometry can be found in \cite{harveylawson} and \cite{djcalibratedgeo}. \\

Given an oriented $k$-dimensional vector subspace $V$ of the tangent space $T_x M$ for some $x \in M$, which we call an \textit{oriented tangent $k$-plane}, there is a Euclidean metric $g|_V$ on $V$. Thus, from $g|_V$ and the orientation on $V$, we obtain a natural volume form $\vol_V$ on $V$. 
\begin{definition}
    A closed $k$-form $\alpha$ is called a \textit{calibration} on $M$ if for all oriented tangent $k$-planes $V$ on $M$, we have
    \begin{align*}
        \alpha|_{V} \leq \vol_V.
    \end{align*}
    We say that an oriented $k$-dimensional submanifold $N$ of $M$ is a \textit{calibrated submanifold} if \begin{align*}
    \alpha|_{T_x N} = \vol_{T_x N},
    \end{align*}
    for all $x \in N$. 
\end{definition}
The examples of calibrated submanifolds that form the focus of this paper are: special Lagrangians in $\C^m$, associatives and coassociatives in $\R^7$, and Cayleys in $\R^8$. We now define each of them. 

\begin{definition}
    Consider $\C^m$ with the complex coordinates $(z_1, \dotsc, z_m)$ with the following standard metric $g$, symplectic form $\omega$, and holomorphic volume form $\Omega$: \begin{align*}
        g &= |dz_1|^2 + \dotsb + |dz_m|^2, \\
        \omega &= \frac{i}{2}(dz_1 \wedge d\overline{z}_1 + \dotsb +dz_m \wedge d\overline{z}_m ),\\
        \Omega &= dz_1 \wedge \dotsb \wedge dz_m. 
    \end{align*}
    Since $\re \Omega$ and $\im \Omega$ are real $m$-forms, \begin{align*}
       \cos\theta \re \Omega+\sin\theta \im \Omega
    \end{align*}
    is a real $m$-form on $\C^m$ and it is a calibration. We say that an $m$-dimensional oriented real submanifold $L$ of $\C^m$ is a \textit{special Lagrangian submanifold} (SL $m$-fold) of $\C^m$ with phase $e^{i \theta}$ for some $\theta \in [0, 2\pi)$ if it is calibrated with respect to $ \cos\theta \re \Omega+\sin\theta \im \Omega$.    
\end{definition}
In this paper, we take $\theta = 0$, which means that we want $L$ to be calibrated by $\re \Omega$. From \cite[Cor. III.1.11]{harveylawson}, we have the following characterization of special Lagrangians which is easier to work with in practice.  
\begin{proposition}
    Let $L$ be a real $m$-dimensional oriented submanifold of $\C^m$. Up to a possible change of orientation, $L$ is a special Lagrangian if and only if \begin{align}
        \omega|_L &= 0
    \end{align}
    and
    \begin{align}    
        \im \Omega|_L &= 0.
    \end{align} 
\end{proposition}

Next, we define associative and coassociative submanifolds of $\R^7$. We use the conventions from \cite{lotay2018calibratedsubmanifolds}.

\begin{definition}
    Let $e_1, \dotsc, e_7$ be an orthonormal basis for $\R^7$ and we denote $e_{i} \wedge e_{j} \wedge \dotsb \wedge e_k$ by $e_{ij\dotsb k}$, where we identify $\R^7$ with $(\R^7)^*$ using the metric. We define a $3$-form $\varphi$ on $\R^7$
    by 
    \begin{align}
    \label{phi}
        \varphi = e_{123}+e_{145}+e_{167}+e_{246}- e_{257}-e_{347}-e_{356}.
    \end{align}
    From \cite[Theorem IV.1.4]{harveylawson}, $\varphi$ is a calibration on $\R^7$ and $3$-dimensional submanifolds $L$ of $\R^7$ calibrated by $\varphi$ are known as \textit{associative $3$-folds}.
\end{definition}
Let $\psi = \star \varphi$ be the $4$-form defined as the Hodge dual of $\varphi$. Then, it is given as 
\begin{align}
\label{psi}
    \psi = e_{4567}+e_{2367}+e_{2345}+e_{1357}-e_{1346}-e_{1256}-e_{1247}.
\end{align}
We can think of $\psi$ as a vector valued $3$-form by using the inner product as follows: 
\begin{align}
    \inp{\psi(V_1, V_2, V_3)}{V_4} = \psi(V_1, V_2, V_3, V_4),
\end{align}
for tangent vectors $V_i$. Using octonionic multiplication, associative $3$-folds can be given an alternative description in terms of $
\psi$ which is more useful (see \cite[Section 2]{Karigiannis_2005} for more details). 
\begin{proposition}{\cite[Proposition 2.3]{Karigiannis_2005}}
\label{associativeschar}
    The subspace spanned by the tangent vectors $V_1, V_2, V_3$ is associative for some orientation if and only if 
    $\psi_i = 0$ for all $i$, where
    \begin{align*}
        \psi_i = \inp{\psi(V_1, V_2, V_3)}{e_i} = \psi(V_1, V_2, V_3,e_i),
    \end{align*}
    for $1 \leq i \leq 7$.
\end{proposition}
Similarly, we now define coassociative submanifolds and give an equivalent characterization of them which is easier to work with. 
\begin{definition}
We know from \cite[Theorem IV.1.16]{harveylawson} that $\psi$ is a calibration. We call $4$-dimensional submanifolds $L$ of $\R^7$ calibrated by $\psi$ as \textit{coassociative $4$-folds}.
\end{definition}
We have the following alternate description from \cite[Proposition IV.4.5 $\&$ Theorem IV.4.6]{harveylawson}
\begin{proposition}
\label{coasociativeprop}
    Up to a change of orientation, a $4$-dimensional submanifold $L$ of $\R^7$ is coassociative if and only if \begin{align*}
        \varphi|_L = 0.
    \end{align*}
\end{proposition}
Finally, we define Cayley $4$-folds in $\R^8$. 
\begin{definition}
     Let $e_0, \dotsc, e_7$ be an orthonormal basis for $\R^8$. Let $\Phi$ be a $4$-form on $\R^8$ defined as follows:
    \begin{align}
    \label{cayleydefinition}
    \begin{split}
        \Phi = e_0 \wedge \varphi + \psi.
    \end{split}
    \end{align}
    Then $\Phi$ is a calibration \cite[Theorem IV.1.24]{harveylawson} and  submanifolds of $\R^8$ calibrated by $\Phi$ are known as \textit{Cayley $4$-folds}.
\end{definition}

With the above convention, we have \begin{align*}
    \Phi = \re(\Omega)+\frac{1}{2}\omega^2,
\end{align*}
where $(\Omega, \omega)$ is the standard $\SU(4)$-structure on $\C^4$.\\

Let us now state some relations between these calibrated submanifolds. We can obtain associative and coassociative submanifolds from special Lagrangians and complex surfaces in the following way (see \cite[Prop 6.5]{lotay2018calibratedsubmanifolds} for details):

\begin{proposition}
\label{assocprop}
Let $\R^7 = \R \times \C^3$. Consider $N = \R \times S \subseteq \R \times \C^3$. Then, 
\begin{enumerate}[(i)]
        \item  $N$ is     associative $\iff$ $S$ is a complex curve.  
        \item $N$ is coassociative $\iff $ $S$ is a special Lagrangian $3$-fold with phase $-i$. 
\end{enumerate}
If $N \subseteq \{0\} \times \C^3$, then 
\begin{enumerate}[(a)]
    \item $N$ is associative $\iff$ $N$ is a special Lagrangian $3$-fold. 
    \item $N$ is coassociative $\iff$ $N$ is a complex surface. 
\end{enumerate}
\end{proposition}

Similarly, we have the following relations for Cayley submanifolds which follow from \eqref{cayleydefinition}  (see \cite[Prop 7.3]{lotay2018calibratedsubmanifolds} for details):
\begin{proposition}
\label{cayleyprop}
    Complex surfaces and special Lagrangian $4$-folds in $\C^4$ are Cayley in $\R^8 = \C^4$. Furthermore, if we write $\R^8 = \R \times \R^7$, then \begin{enumerate}[(i)]
        \item $\R \times S$ is Cayley $\iff$ $S$ is associative in $\R^7$.
        \item $N \subseteq\R^7$ is Cayley in $\R^8$ $\iff$ $N$ is coassociative in $\R^7$. 
    \end{enumerate}
\end{proposition}

\section{Method}
\label{method}

 In this section, we give a heuristic description of the technique we use to construct the calibrated submanifolds with symmetries. Since we are only dealing with special Lagrangians, associatives, coassociatives and Cayleys in Euclidean space, we take the manifold $M$ to be either $\C^n$, $\R^7$ or $\R^8$. Then, we consider an action of a Lie group $G$ on $M$ which preserves the calibration form. Let us assume that the orbits of this action are parametrized by $\theta_1, \dotsc, \theta_k \in \R$ where $k \leq n$, where $n$ is the dimension of the manifold.\\

 Let us now take  $$\alpha: (-\epsilon_1, \epsilon_1) \times \dotsb \times (-\epsilon_l, \epsilon_l) \rightarrow M$$
 to be a smooth map where $\epsilon_i > 0$. In the case of the $\bT^2$-invariant coassociatives, it is a surface $\alpha(s, t)$ in $\R^7$ whereas in the rest of the examples it is a curve $\alpha(t)$. \\

 We package the information about the group acting on $\alpha$ through the following map 
 \begin{align*}
     F: (-\epsilon_1, \epsilon_1) \times \dotsb \times (-\epsilon_l, \epsilon_l) \times  \mathbb{R}^{k} \rightarrow M,
 \end{align*}
given as
 \begin{align*}
     F(t_1, \dotsc, t_l, \theta_1, \dotsc, \theta_k) = A_{\theta_1, \dotsc, \theta_k} \cdot \alpha(t_1, \dotsc, t_l),
 \end{align*}
where $A_{\theta_1, \dotsc, \theta_k} \cdot \alpha(t_1, \dotsc, t_l)$ is just the point on the orbit through $\alpha(t_1, \dotsc, t_l)$ corresponding to the coordinates $\theta_1, \dotsc, \theta_k$ for $A_{\theta_1, \dotsc, \theta_k} \in G$.  \\

Then, we define the following vector fields along $F$:
\begin{align*}
    V_i = \frac{\partial F}{\partial \theta_i},
\end{align*}
for $1 \leq i \leq k$, and 
\begin{align*}
    V_j =  \frac{\partial F}{\partial t_j},
\end{align*}
for $k+1 \leq j \leq l$. \\

Next, we want to ensure that $F$ is an immersion, and through that process, we obtain certain necessary conditions that $\alpha$ has to satisfy. Once we impose those conditions on $\alpha$, we demand that the tangent space to the submanifold determined by $F$ be calibrated, to obtain differential equations for $\alpha$. Finally, we solve these differential equations to obtain our $G$-invariant calibrated submanifolds. 
 
\section{$\bT^{n-1}$-invariant special Lagrangian $n$-folds} \label{maxtorusSL}
In this section, we construct special Lagrangian $n$-folds which are invariant under the action of a maximal torus $\bT^{n-1}$ in $\SU(n)$ (see \cite[Chapter VII.2]{helgason1979differential} for more details on maximal tori of Lie groups) using the method described in Section \ref{method}. Harvey-Lawson  \cite[Section III.3.A]{harveylawson} previously constructed these SL $n$-folds using an implicit characterization of special Lagrangians. \\

Let $\bT^{n-1}$ be a maximal torus in $\SU(n)$ given as 
\begin{align*}
     \bT^{n-1} = \{\diag(e^{i\theta_1}, \dotsc, e^{i\theta_n}): \theta_1+ \dotsb +\theta_n = 0\}.
\end{align*}

Now, to look at the orbits of its action on $\C^n$, let us take $z \in \C^n$ satisfying $z_k \neq 0$ for all $k$ and consider the orbit $\bT^{n-1} \cdot z$. For $g \in \bT^{n-1}$, we have that if $g \cdot z = z$, then for all $k$, 
\begin{align*}
     e^{i\theta_k} z_k = z_k \implies e^{i\theta_k} = 1 \implies \theta_k = 0,
\end{align*}
as $z_k \neq 0$. That is, $g \cdot z = z \iff g = \Id$. Thus,  the orbit $\bT^{n-1} \cdot z$ of $z$ is equivariantly diffeomorphic to $\bT^{n-1}$.\\

So let us fix a curve $\alpha: (-\epsilon, \epsilon ) \rightarrow \C^n$ where $\alpha_k(t) \neq 0$ for $1 \leq k\leq n$ for all time $t$. Then the orbit of $\alpha(t)$ under $\bT^{n-1}$ is given as \begin{align*}
    \bT^{n-1}\cdot \alpha(t) =  \{(e^{i\theta_1}\alpha_1(t), \dotsc, e^{i\theta_{n-1}}\alpha_{n-1}(t),e^{i\theta_{n}}\alpha_{n}(t) ): t \in (-\epsilon, \epsilon)\}.
\end{align*}
Let us now define $F: (-\epsilon, \epsilon) \times \R^{n-1} \rightarrow \C^n$ as 
\begin{align*}
    F(t, \theta_1, \dotsc, \theta_{n-1}) = A_{\theta_1, \dotsc, \theta_{n-1} }\cdot \alpha(t),
\end{align*}
where $A_{\theta_1, \dotsc, \theta_{n-1}} = \diag(e^{i\theta_1}, \dotsc, e^{i\theta_{n-1}}, e^{-i(\theta_1+\dotsb + \theta_{n-1})})$. We now deduce conditions which ensure that $F$ is an immersion.

\begin{lemma}
    Let us assume that there exists at least one $k$ such that for all $t \in (-\epsilon, \epsilon)$, we have $\alpha_k(t) \neq e^{i\lambda t}$ for any $\lambda \in \R$. Then, the map $F$ is an immersion. That is, the vector fields 
    \begin{align*}
       V_{k} := \frac{\pr F}{\pr \theta_k} = (0, \dotsc, i e^{i\theta_k}\alpha_k, \dotsc, 0, -ie^{-i(\theta_1+\dotsb + \theta_{n-1})}\alpha_n),
    \end{align*}
and 
    \begin{align*}
       V_n := \frac{\pr F}{\pr t} = (e^{i\theta_1}\alpha_1'(t), \dotsc, e^{i\theta_{n-1}}\alpha_{n-1}'(t), e^{-i(\theta_1 +\dotsb+\theta_{n-1})}\alpha_{n}'(t)),
    \end{align*}
for $1 \leq k \leq n-1$, are linearly independent. 
\end{lemma}
\begin{proof}
    Note that the $V_k$ are linearly independent for $1 \leq k \leq n-1$.
For the rest, if we take $b, c_k \in \R$ and \begin{align*}
    \sum_k c_k \frac{\pr F}{\pr \theta_k} + b \frac{\pr F}{\pr t} = 0,
\end{align*}
then if $b = 0$, we get that $c_k = 0$ for all $k$. But if $b \neq 0$, then 
\begin{align*}
    \alpha'_k = -i\frac{c_k}{b} \alpha_k
\end{align*}
and \begin{align*}
    \alpha'_n = -i\frac{(c_1 +\dotsb+c_{n-1})}{b} \alpha_n.
\end{align*}
So if $b \neq 0$, then for all $k \in 1, \dotsc, n-1$ we have $\alpha'_k = i \lambda_k \alpha_k$ for some $\lambda_k \in \R$. So if we assume that at least one $k$ is such that $\alpha'_k \neq i \lambda \alpha_k \implies \alpha_k(t) \neq e^{i\lambda t}$ for any $\lambda \in \R$ and for all $t \in (-\epsilon, \epsilon)$, then we get linear independence for $\{V_1, \dotsc, V_n\}$ and thus $F$ is an immersion in such case. 
\end{proof}

Let us set $\alpha_k(t) = u_k(t) + iv_k(t)$. Note that we have the expansion
\begin{align*}
    ie^{i\theta_k}\alpha_k 
 &= i(\cos\theta_k + i \sin\theta_k)(u_k + iv_k)\\ &= (-v_k\cos\theta_k -u_k\sin\theta_k) + i(u_k\cos\theta_k -v_k\sin\theta_k ).
\end{align*}
Hence, we can write the vector fields as 
\begin{align*}
    V_k  &= (-v_k\cos\theta_k -u_k\sin\theta_k)\frac{\pr}{\pr x^k} + (u_k \cos\theta_k -v_k\sin\theta_k )\frac{\pr}{\pr y^k}\\ &\quad+(v_n\cos\theta_n+u_n\sin\theta_n)\frac{\pr}{\pr x^n}
    +(-u_n\cos\theta_n+v_n\sin\theta_n) \frac{\pr}{\pr y^n},
\end{align*}
and 
\begin{align*}
    V_n &= (u'_1 \cos\theta_1 -v'_1\sin\theta_1)\frac{\pr}{\pr x^1} + (u'_1\sin\theta_1 + v'_1\cos\theta_1)\frac{\pr}{\pr y^1} \\
    &\quad +\dotsb + (u'_{n} \cos\theta_n -v'_{n}\sin\theta_n)\frac{\pr}{\pr x^n} + (u'_{n}\sin\theta_n + v'_{n}\cos\theta_n)\frac{\pr}{\pr y^n}.
\end{align*}
\begin{proposition}
\label{tn-1sl}
    The subset formed by $(\alpha_1, \dotsc, \alpha_n) \in \C^n$ such that for $1 \leq k \leq n-1$, 
    \begin{align*}
        |\alpha_k|^2 - |\alpha_n|^2 = c_k
    \end{align*}
    for some $c_k \in \R$, and 
    \begin{align*}
    \begin{cases}
        \im(\alpha_1 \dotsb \alpha_n) = c, & \text{if $n$ is odd,}\\
        \re(\alpha_1 \dotsb \alpha_n) = c, & \text{if $n$ is even,}
    \end{cases}
    \end{align*}
    for some $c \in \R$, forms a special Lagrangian submanifold in $\C^n$. 
\end{proposition}
\begin{proof}

The $2$-form $\omega$ is 
\begin{align*}
    \omega = dx^1 \wedge dy^1 + \dotsb+ dx^n \wedge dy^n.
\end{align*}
For $k \neq l$, we have \begin{align*}
    \omega(V_k, V_l) &= (dx^n \wedge dy^n)(V_k, V_l)\\
    &= dx^n(V_k) dy^n(V_l) - dx^n(V_l) dy^n(V_k)\\
    &= (v_n\cos\theta_n+u_n\sin\theta_n)(-u_n\cos\theta_n+v_n\sin\theta_n) \\&\quad- (v_n\cos\theta_n+u_n\sin\theta_n)(-u_n\cos\theta_n+v_n\sin\theta_n)= 0
\end{align*}
and 
\begin{align*}
    \omega(V_k, V_n) &= (dx^k\wedge dy^k + dx^n \wedge dy^n)(V_k, V_n)\\
    &= -(v_k\cos\theta_k +u_k\sin\theta_k)(u'_{k}\sin\theta_k + v'_{k}\cos\theta_k)\\&\quad - (u_k\cos\theta_k -v_k\sin\theta_k )(u'_k \cos\theta_k -v'_{k}\sin\theta_k)\\
    &\quad+ (v_n\cos\theta_n+u_n\sin\theta_n)(u'_{n}\sin\theta_n + v'_{n}\cos\theta_n)\\&\quad+(u_n\cos\theta_n-v_n\sin\theta_n)(u'_{n} \cos\theta_n -v'_{n}\sin\theta_n)\\
    &= -v_kv'_k -u_ku'_k +v_n v'_n + u_nu'_n.
\end{align*}
As $|\alpha_k|^2 = u_k^2+v_k^2$, we have that 
\begin{align*}
    F^*\omega = 0 &\iff  (|\alpha_k|^2 )'- (|\alpha_n|^2)' = 0\\
    &\iff |\alpha_k|^2 = |\alpha_n|^2 +c_k,
\end{align*}
for some $c_k \in \R$ for each $k = 1, \dotsc, n-1$. Now, to get the condition for this Lagrangian submanifold of $\C^n$ to be ``special" first note that 
\begin{align*}
    (dz^k)(V_k) &= i e^{i \theta_k} \alpha_k,\\
    (dz^n)(V_k) &= -i e^{i\theta_n}\alpha_n,\\
    (dz^k)(V_n) &= e^{i\theta_k}\alpha'_k, \\
    (dz^n) (V_n) &= e^{i \theta_n}\alpha'_n. 
\end{align*}
Therefore,  as $\theta_1 \dotsb + \theta_n = 0$, from above we have 
\begin{align*}
     (dz^1 \wedge \dotsb \wedge dz^n)(V_1, \dotsc, V_n) &= i^{n-1}\alpha_1\dotsb\alpha'_n +\sum_{k = 1}^{n-1} i^{n-1} \alpha_1 \dotsb \alpha'_k\dotsb \alpha_n \\  
     &= \sum_{k = 1}^{n} i^{n-1} \alpha_1 \dotsb \alpha'_k\dotsb \alpha_n .
\end{align*}
If $n$ is odd, then $i^{n-1} = \pm 1$ and hence 
\begin{align*}
    (\im(dz^1 \wedge \dotsb \wedge dz^n))(V_1, \dotsc, V_n) = 0 &\iff \im\left( \sum_{k = 1}^{n} \alpha_1 \dotsb \alpha'_k\dotsb \alpha_n\right) = 0\\
    &\iff \im(\alpha_1 \dotsb \alpha_n)' = 0\\
    &\iff \im(\alpha_1 \dotsb \alpha_n) = c,
\end{align*}
for some $c \in \R$. If $n$ is even, then 
$i^{n-1} = \pm i$ and hence 
\begin{align*}
    (\im(dz^1 \wedge \dotsb \wedge dz^n))(V_1, \dotsc, V_n) = 0 &\iff \re\left( \sum_{k = 1}^{n} \alpha_1 \dotsb \alpha'_k\dotsb \alpha_n\right) = 0\\
    &\iff \re(\alpha_1 \dotsb \alpha_n)' = 0\\
    &\iff \re(\alpha_1 \dotsb \alpha_n) = c,
\end{align*}
for some $c \in \R$.
\end{proof}
\section{$\SO(n)$-invariant special Lagrangian $n$-folds}
\label{soninvariant}

We now construct special Lagrangian $n$-folds invariant under the diagonal action of $\SO(n) \subset \SU(n)$ using this technique. Like the SL $n$-fold in the previous section, these submanifolds were also previously constructed by Harvey-Lawson \cite[Section III.3.B]{harveylawson} where they used an explicit characterization of special Lagrangians in terms of the graph of a function. \\

The diagonal action of $\SO(n) \subset \SU(n)$ on $\C^{n} = \R^n \oplus \R^n$ is given as \begin{align*}
    \begin{bmatrix}
        A & 0 \\
        0 & A 
    \end{bmatrix},
\end{align*}
where $A \in \SO(n)$. Fix $z = x+i\lambda x \in \C^n$ for $x \in \R^n$ and $\lambda \in \R$. If $z \neq 0$, then we have that the orbit $\SO(n) \cdot z$ is equivariantly diffeomorphic to $S^{n-1}$. That is, \begin{align*}
    \SO(n)\cdot (x, \lambda x) = \{(Ax, \lambda Ax): A \in \SO(n)\}
\end{align*}
where each of the above factors is equivariantly diffeomorphic to $S^{n-1}_{|x|}$ and $S^{n-1}_{\lambda|x|}$, where the subscript denotes a radius of $|x|$ and $\lambda |x|$ respectively. Now, let $\alpha = u+iv : (-\epsilon, \epsilon) \rightarrow \C^n$ be a smooth curve in $\C^n$ and suppose $\alpha(t) \neq 0$ for all $t$. Then, let us define $F: (-\epsilon, \epsilon) \times S^{n-1} \rightarrow \C^n$ as 
\begin{align*}
    F(t,p) = (|u(t)|p, |v(t)|p).
\end{align*}
We next give a condition for $F$ to be an immersion.
\begin{lemma}
    Let us assume that both $u(t)$ and $v(t)$ are not tangent to $S^{n-1}$ for all $t \in (-\epsilon, \epsilon)$. Then, $F$ is an immersion. 
\end{lemma}
\begin{proof}
    Let us denote $\alpha(t) = (\alpha_1(t), \alpha_2(t))$ where $\alpha_1 = u$ and $\alpha_2 = v$. Then, we have 
    \begin{align*}
        F = F_1 \oplus F_2,
    \end{align*}
    where $F_i = |\alpha_i(t)|p$. Note that each $F_i$ can be thought of as a map 
        \[\begin{tikzcd}
    	{(-\epsilon, \epsilon) \times S^{n-1}} & {\R^+ \times S^{n-1} } & {\R^n}\setminus \{0\}
    	\arrow["h", from=1-1, to=1-2]
    	\arrow["\cong", from=1-2, to=1-3]
        \end{tikzcd}\] 
        given as \begin{align*}
                (t, p) \mapsto (|\alpha_i (t)|, p) \mapsto |\alpha_i(t)|p.
        \end{align*}
        Then, the differential of $h$ is given as 
        \begin{align*}  
           h_* = \left[\begin{array}{c|c}
           \frac{1}{|\alpha_i(t)|}\alpha_i(t) \cdot \alpha'_i(t) & 0 \\
           \hline
                0 & |\alpha_i(t)| \Id  \\
            \end{array}\right].
        \end{align*}
        Hence, if $\alpha_i$ is not tangent to $S^{n-1}$, then $ \frac{1}{|\alpha_i(t)|}\alpha_i(t) \cdot \alpha'_i(t) \neq 0$ and thus $h_*$ is injective. Therefore, since each $(F_i)_*$ is injective, $F$ is an immersion.      
\end{proof}

Now, let $e_1, \dotsc, e_{n-1}$ be an orthonormal frame for $S^{n-1}$ at $p \in S^{n-1}$. With respect to the coordinate frame on $\R^n$, let us write $e_k = \sum_i c_{ki} \pr_i$ for some $c_{ki} \in \R$. Then, we have \begin{align*}
    (F_*)(e_k) = (|u(t)| e_k, |v(t)| e_k)
\end{align*}
and 
\begin{align*}   
    (F_*)\left(\frac{\pr}{\pr t}\right) = \left(\frac{u(t)}{|u(t)|} \cdot u'(t)p, \frac{v(t)}{|v(t)|} \cdot v'(t)p\right).
\end{align*}
Let $V_k := (F_*)(e_k) $ for $k = 1, \dotsc, n-1$ and let $V_n := (F_*)\left(\frac{\pr}{\pr t}\right)$. Note that \begin{align*}
    V_k = \left(c_{k1}|u(t)| \frac{\pr}{\pr x_1} + c_{k1}|v(t)| \frac{\pr}{\pr y_1}\right) + \dotsb + \left(c_{kn}|u(t)| \frac{\pr}{\pr x_n} + c_{kn}|v(t)| \frac{\pr}{\pr y_n}\right)
\end{align*} 
and 
\begin{align*}
    V_n &= \left(\frac{u(t)}{|u(t)|}\cdot u'(t)x_1\frac{\pr}{\pr x_1 }+ \frac{v(t)}{|v(t)|}\cdot v'(t)y_1\frac{\pr}{\pr y_1 }\right)+\dotsb\\&\quad\dotsb+ \left(\frac{u(t)}{|u(t)|}\cdot u'(t)x_n\frac{\pr}{\pr x_n}+ \frac{v(t)}{|v(t)|}\cdot v'(t)y_n\frac{\pr}{\pr y_n }\right).
\end{align*}
\\
\begin{proposition}
    Let $\alpha(t) = u(t) + iv(t)$ be a curve as defined above. Then, if $c \in \R$, the set
    \begin{align*}
       \{ \alpha(t) \in \C^n:  \im((|u(t)|  +i|v(t)|)^{n}) = c\}
    \end{align*}
    forms a special Lagrangian submanifold in $\C^n$.    
\end{proposition}
\begin{proof}

For $1 \leq k \neq l \leq n$, we have 
\begin{align*}
   \omega(V_k, V_l) &= (dx^1 \wedge dy^1 + \dotsb + dx^n \wedge dy^n)(V_k, V_l)\\
   &= (c_{k1}|u(t)|c_{l1}|v(t)| - c_{k1}|v(t)|c_{l1}|u(t)|) + \dotsb 
   \\&\quad+ (c_{kn}|u(t)|c_{ln}|v(t)| - c_{kn}|v(t)|c_{ln}|u(t)|)\\
   &= 0,
\end{align*}
and note that since from our setup above, we have the correspondence between coordinates on $\R^n$ and $\R^n \oplus \R^n$ given as $(s_1, \dotsc, s_n) \mapsto (x_1, \dotsc, x_n, y_1, \dotsc, y_n)$ where $x_i = y_i = s_i$, we have 
\begin{align*}
    \omega(V_k, V_n) &= \left(c_{k1}|u(t)|\frac{v(t)}{|v(t)|}\cdot v'(t)y_1-c_{k1}|v(t)|\frac{u(t)}{|u(t)|}\cdot u'(t)x_1\right) + \dotsb \\&\quad\dotsb+ \left(c_{kn}|u(t)|\frac{v(t)}{|v(t)|}\cdot v'(t)y_n-c_{kn}|v(t)|\frac{u(t)}{|u(t)|}\cdot u'(t)x_n\right)\\
    &= \left(|u(t)|\frac{v(t)}{|v(t)|}\cdot v'(t)-|v(t)|\frac{u(t)}{|u(t)|}\cdot u'(t)\right)(\sum_{i} c_{ki}s_i).
\end{align*}

 As $e_1, \dotsc, e_{n-1}$ are perpendicular to $p$, we have that $e_k \cdot p = \sum_i c_{ki} s_i = 0$. Hence, we get 
 \begin{align*}
    \omega(V_k, V_n) = 0.
\end{align*}
Now, we want to check the ``special" condition. We have for $1 \leq j \leq n-1,$
\begin{align*}
    dz^k(V_j) &= (dx^k + idy^k)\bigg(\left(c_{j1}|u(t)| \frac{\pr}{\pr x_1} + c_{j1}|v(t)| \frac{\pr}{\pr y_1}\right) +\dotsb\\&\quad\dotsb+ \left(c_{jn}|u(t)| \frac{\pr}{\pr x_n} + c_{jn}|v(t)| \frac{\pr}{\pr y_n}\right)\bigg)\\
    &= c_{jk}(|u(t)|  +i|v(t)|) 
\end{align*}
and \begin{align*}
    dz^k(V_n) &= (dx^k + idy^k)\bigg(\left(\frac{u(t)}{|u(t)|}\cdot u'(t)x_1\frac{\pr}{\pr x_1 }+ \frac{v(t)}{|v(t)|}\cdot v'(t)y_1\frac{\pr}{\pr y_1 }\right)+\dotsb\\&\quad\dotsb+ \left(\frac{u(t)}{|u(t)|}\cdot u'(t)x_n\frac{\pr}{\pr x_n}+ \frac{v(t)}{|v(t)|}\cdot v'(t)y_n\frac{\pr}{\pr y_n }\right)\bigg)\\
    &= \frac{u(t)}{|u(t)|}\cdot u'(t)x_k+ i \frac{v(t)}{|v(t)|}\cdot v'(t)y_k\\
    &= \left( \frac{u(t)}{|u(t)|}\cdot u'(t)+i \frac{v(t)}{|v(t)|}\cdot v'(t)\right)s_k.
    \end{align*}

     Then, we have \begin{align*}
        (dz_1 \wedge \dotsb \wedge dz_n)(V_1, \dotsc, V_n) &= \det(dz^k(V_j))\\
        &= \det\left(\begin{bmatrix}
            c_{jk}(|u(t)|  +i|v(t)|) & \left( \frac{u(t)}{|u(t)|}\cdot u'(t)+i \frac{v(t)}{|v(t)|}\cdot v'(t)\right)s_k
    \end{bmatrix}\right)\\
    &= \det\left(\begin{bmatrix}
            c_{jk}(|u(t)|  +i|v(t)|) & \left( (|u(t)|  +i|v(t)|)' \right)s_k
    \end{bmatrix}\right)\\
    &= (|u(t)|  +i|v(t)|)^{n-1} (|u(t)|  +i|v(t)|)'\det\left(\begin{bmatrix}
            C &\mid & \vec{s}
    \end{bmatrix}\right),
    \end{align*}
   where $C$ is the matrix with the entries $c_{jk}$ and $\vec{s}$ is the vector with entries $s_k$. Note that $(\begin{bmatrix}
            C &\mid & \vec{s}
    \end{bmatrix})$ is the change of basis matrix from the coordinate basis for $\R^n$ to the basis $$e_1, \dotsc, e_{n-1}, p = \sum_l x_l\frac{\pr}{\pr x_l}.$$ Hence, we have $\det(\begin{bmatrix}
            C &\mid & \vec{s}
    \end{bmatrix}) \neq 0$. Thus, we get that the $V_i$ span a special Lagrangian subspace iff \begin{align*}       &\quad\quad\quad\im((dz_1 \wedge \dotsb \wedge dz_n)(V_1, \dotsc, V_n)) = 0 \\&\iff \im((|u(t)|  +i|v(t)|)^{n-1} (|u(t)|  +i|v(t)|)') = 0\\
        &\iff \im(((|u(t)|  +i|v(t)|)^{n}))' = 0\\
        &\iff \im((|u(t)|  +i|v(t)|)^{n}) = c,
    \end{align*}
for some constant $c \in \R$. 
\end{proof}

\section{$\mathbb{T}^2$-invariant associative $3$-folds}
\label{t2invariantassoc}
We now consider the action of a maximal torus in $\Gt$ on $\R^7$ in this section and show that a $\mathbb{T}^2$-invariant associative $3$-fold must be a special Lagrangian $3$-fold in $\C^3$. Since $\SU(3) \subset \Gt$, and both have rank two, any maximal torus for $\SU(3)$ is also a maximal torus for $\Gt$. Thus, with respect to the decomposition $\R^7 = \R \oplus \C^3$, we know that the elements $X$ in the Lie algebra of a maximal torus $\mathbb{T}^2 \subset \Gt$ are of the form 
\begin{align*}
    X_{\theta_1, \theta_2} = 
    \begin{bmatrix}
        0 & & & & & & \\
        & 0& -\theta_1& & & \\
        &\theta_1 &0 & & & \\
         & & & 0 &-\theta_2 & \\
         & & & \theta_2&0 & &\\
         & & &  & &0&-\theta_3\\
         & & &  & &\theta_3&0\\
    \end{bmatrix}
\end{align*}
where $\theta_1, \theta_2, \theta_3 \in \R$ and $\theta_1 + \theta_2 + \theta_3 = 0$. Note that as 
\begin{align*}
    \exp\left({\begin{bmatrix}
        0 & -\theta\\
        \theta & 0 \\
    \end{bmatrix}}\right) = \begin{bmatrix}
        \cos\theta & -\sin\theta\\
        \sin\theta & \cos\theta
    \end{bmatrix},
\end{align*}
we have that the exponential of $X_{\theta_1, \theta_2}$ is given as
\begin{align*}
    e^{X_{\theta_1, \theta_2}} = \begin{bmatrix}
        1 & & & & & & \\
        & \cos\theta_1&-\sin\theta_1 & & &\\
        &\sin\theta_1 &\cos\theta_1 & & & \\
         & & & \cos\theta_2 &-\sin\theta_2 & \\
         & & & \sin\theta_2&\cos\theta_2 & &\\
         & & &  & &\cos\theta_3&-\sin\theta_3\\
         & & &  & &\sin\theta_3&\cos\theta_3\\
    \end{bmatrix}.
\end{align*}
Let $\alpha: (-\epsilon, \epsilon) \rightarrow \R^7$ be a smooth curve in $\R^7$. The orbit of $\alpha(t)$ is then given as \begin{align*}
    e^{X_{\theta_1, \theta_2}}\cdot \alpha(t) &= \begin{bmatrix}
        1 & & & & & & \\
        & \cos\theta_1&-\sin\theta_1 & & &\\
        &\sin\theta_1 &\cos\theta_1 & & & \\
         & & & \cos\theta_2 &-\sin\theta_2 & \\
         & & & \sin\theta_2&\cos\theta_2 & &\\
         & & &  & &\cos\theta_3&-\sin\theta_3\\
         & & &  & &\sin\theta_3&\cos\theta_3\\
    \end{bmatrix}\begin{bmatrix}
        \alpha_1 \\
        \alpha_2 \\
        \alpha_3 \\\alpha_4\\\alpha_5 \\\alpha_6 \\ \alpha_7 
    \end{bmatrix}\\
    &= \begin{bmatrix}
       \alpha_1 \\
        \alpha_2 \cos\theta_1 -\alpha_3\sin\theta_1 \\
        \alpha_2\sin\theta_1 +\alpha_3\cos\theta_1 \\\alpha_4\cos\theta_2 -\alpha_5\sin\theta_2\\ \alpha_4 \sin\theta_2 + \alpha_5 \cos\theta_2\\ \alpha_6\cos\theta_3 -\alpha_7 \sin\theta_3 \\ 
        \alpha_6 \sin\theta_3 +\alpha_7 \cos\theta_3
    \end{bmatrix} =  \begin{bmatrix}
        \alpha_1\\ e^{i\theta_1}\cdot z_1\\
        e^{i\theta_2}\cdot z_2\\
        e^{i\theta_3}\cdot z_3
    \end{bmatrix},
\end{align*} 
where we define $z_1 = \alpha_2 + i \alpha_3$, $z_2 = \alpha_4 + i \alpha_5$ and $z_3 = \alpha_6 + i\alpha_7$. From above, we can see that the orbit through a point is $2$-dimensional if and only if at most one of $z_1, z_2, z_3$ is zero.\\

Then, we define $F: (-\epsilon, \epsilon) \times  \R^2 \rightarrow \R^7$ as \begin{align*}
     F(t, \theta_1, \theta_2) = e^{X_{\theta_1, \theta_2}}\cdot \alpha(t).
\end{align*}

Substituting from above, we get
\begin{align}
\label{maxtorusmap}
    F(t, \theta_1, \theta_2) 
    &= \begin{bmatrix}
       \alpha_1 \\
        \alpha_2 \cos\theta_1 -\alpha_3\sin\theta_1 \\
        \alpha_2\sin\theta_1 +\alpha_3\cos\theta_1 \\\alpha_4\cos\theta_2 -\alpha_5\sin\theta_2\\ \alpha_4 \sin\theta_2 + \alpha_5 \cos\theta_2\\ \alpha_6\cos\theta_3 -\alpha_7 \sin\theta_3 \\ 
        \alpha_6 \sin\theta_3 +\alpha_7 \cos\theta_3
    \end{bmatrix}.
\end{align}
Hence, we have the vector fields
\begin{align*}
    V_1 := \frac{\pr F}{\pr \theta_1} &= (0, -\alpha_2\sin\theta_1-\alpha_3\cos\theta_1, \alpha_2 \cos\theta_1-\alpha_3\sin\theta_1, 0, 0, \alpha_6\sin\theta_3+\alpha_7 \cos\theta_3, \\&\quad -\alpha_6\cos\theta_3  +\alpha_7 \sin\theta_3 ),\\
    V_2 := \frac{\pr F}{\pr \theta_2} &= (0, 0, 0, -\alpha_4\sin\theta_2-\alpha_5 \cos\theta_2 , \alpha_4\cos\theta_2-\alpha_5 \sin\theta_2, \alpha_6\sin\theta_3+\alpha_7 \cos\theta_3, \\&\quad -\alpha_6\cos\theta_3  +\alpha_7 \sin\theta_3),\\
    V_3 := \frac{\pr F}{\pr t} &= (\alpha'_1,  \alpha'_2 \cos\theta_1 -\alpha'_3\sin\theta_1, \alpha'_2\sin\theta_1 +\alpha'_3\cos\theta_1, \alpha'_4\cos\theta_2 -\alpha'_5\sin\theta_2, \\&\quad  \alpha'_4 \sin\theta_2 + \alpha'_5 \cos\theta_2,  \alpha'_6\cos\theta_3 -\alpha'_7 \sin\theta_3,
        \alpha'_6 \sin\theta_3 +\alpha'_7 \cos\theta_3).
\end{align*}

From the characterization in Proposition \ref{associativeschar}, by evaluating the vector fields $V_1, V_2, V_3$ using \eqref{psi} through direct computation, we obtain the following equations $\psi_i = \psi(V_1, V_2, V_3, e_i) = 0$:
\begin{align*}
   \psi_1 &= (-\alpha_2\alpha_5\alpha_6 - \alpha_3\alpha_4\alpha_6-\alpha_2\alpha_4\alpha_7+\alpha_3\alpha_5\alpha_7)' = 0, \\
    \psi_2 &= ((-\alpha_4\alpha_6 + \alpha_5 \alpha_7)\alpha'_1 - (-\alpha_6\alpha'_6 -\alpha_7\alpha'_7+\alpha_4\alpha'_4+\alpha_5\alpha'_5)\alpha_3)\sin\theta_1 \\&\quad+ ((\alpha_4\alpha_7+\alpha_5\alpha_6)\alpha'_1 +(-\alpha_6\alpha'_6 -\alpha_7\alpha'_7+\alpha_4\alpha'_4+\alpha_5\alpha'_5)\alpha_2)\cos\theta_1 = 0,
\end{align*}
\begin{align*}
     \psi_3 &= ((\alpha_4\alpha_7 + \alpha_5 \alpha_6)\alpha'_1 + (-\alpha_6\alpha'_6 -\alpha_7\alpha'_7+\alpha_4\alpha'_4+\alpha_5\alpha'_5)\alpha_2)\sin\theta_1 \\&\quad+ ((\alpha_4\alpha_6-\alpha_5\alpha_7)\alpha'_1 +(-\alpha_6\alpha'_6 -\alpha_7\alpha'_7+\alpha_4\alpha'_4+\alpha_5\alpha'_5)\alpha_3)\cos\theta_1 = 0,
\end{align*}
\begin{align*}
    \psi_4 &= ((-\alpha_2\alpha_6 +\alpha_3 \alpha_7)\alpha'_1 - (\alpha_6\alpha'_6 +\alpha_7\alpha'_7-\alpha_2\alpha'_2-\alpha_3\alpha'_3)\alpha_5)\sin\theta_2 \\&\quad+ ((\alpha_2\alpha_7+\alpha_3\alpha_6)\alpha'_1 +(\alpha_6\alpha'_6 +\alpha_7\alpha'_7-\alpha_2\alpha'_2-\alpha_3\alpha'_3)\alpha_4)\cos\theta_2 = 0,
\end{align*}
\begin{align*}
    \psi_5 &= ((\alpha_2\alpha_7+\alpha_3\alpha_6)\alpha'_1 + (\alpha_6\alpha'_6 +\alpha_7\alpha'_7-\alpha_2\alpha'_2-\alpha_3\alpha'_3)\alpha_4)\sin\theta_2 \\&\quad+ ((\alpha_2\alpha_6 -\alpha_3 \alpha_7)\alpha'_1 +(\alpha_6\alpha'_6 +\alpha_7\alpha'_7-\alpha_2\alpha'_2-\alpha_3\alpha'_3)\alpha_5)\cos\theta_2 = 0,
\end{align*}
\begin{align*}
    \psi_6 &= ((\alpha_3\alpha_5 -\alpha_2 \alpha_4)\alpha'_1 + (\alpha_4\alpha'_4 +\alpha_5\alpha'_5-\alpha_2\alpha'_2-\alpha_3\alpha'_3)\alpha_7)\sin\theta_3 \\&\quad+ ((\alpha_2\alpha_5+\alpha_3\alpha_4)\alpha'_1 -(\alpha_4\alpha'_4+\alpha_5\alpha'_5-\alpha_2\alpha'_2-\alpha_3\alpha'_3)\alpha_6)\cos\theta_3 = 0,
\end{align*}
\begin{align*}
    \psi_7 &= ((\alpha_2\alpha_5 +\alpha_3 \alpha_4)\alpha'_1 - (\alpha_4\alpha'_4 +\alpha_5\alpha'_5-\alpha_2\alpha'_2-\alpha_3\alpha'_3)\alpha_6)\sin\theta_3 \\&\quad+ ((\alpha_2\alpha_4-\alpha_3\alpha_5)\alpha'_1 -(\alpha_4\alpha'_4+\alpha_5\alpha'_5-\alpha_2\alpha'_2-\alpha_3\alpha'_3)\alpha_7)\cos\theta_3 = 0.
\end{align*}
Rewriting the above equations in terms of complex coordinates $z_1, z_2, z_3,$ as defined above, we have the following equivalent system of equations:

\begin{align}
     \im(z_1z_2z_3)' &= 0, \label{eqn1}\\
     \im(z_2z_3) \alpha'_1 + \frac{1}{2}(|z_2|^2-|z_3|^2)'\re(z_1) &= 0, \label{eqn2}\\
      \re(z_2z_3)\alpha'_1 +\frac{1}{2}(|z_2|^2-|z_3|^2)'\im(z_1) &= 0, \label{eqn3}\\
     \im(z_1z_3) \alpha'_1 + \frac{1}{2}(|z_3|^2-|z_1|^2)'\re(z_2) &= 0, \label{eqn4}\\
    \re(z_1z_3) \alpha'_1 + \frac{1}{2}(|z_3|^2-|z_1|^2)'\im(z_2) &= 0, \label{eqn5}\\
    \im(z_1z_2) \alpha'_1 + \frac{1}{2}(|z_1|^2-|z_2|^2)'\re(z_3) &= 0, \label{eqn6}\\
    \re(z_1z_2) \alpha'_1 + \frac{1}{2}(|z_1|^2-|z_2|^2)'\im(z_3) &= 0. \label{eqn7}
\end{align}
Now, we show that every such associative submanifold has to be a special Lagrangian in $\C^3$.
\begin{theorem}
\label{assocslthm}
      Let $N \subseteq \R^7$ be invariant under the action of a maximal torus $\mathbb{T}^2 \subset \Gt$. Then \begin{align*}
          \text{N is an associative in $\R^7$ $\iff$ $N$ is a special Lagrangian in $\C^3$.}
      \end{align*}
\end{theorem}
\begin{proof}
    Let us take $z_1(t),z_2(t) \neq 0$ for all $t$. First, suppose  $\alpha_1' = 0$ in the equations \eqref{eqn1}-\eqref{eqn5} above. Note that this forces $z_3(t) \neq 0$ for all $t$, because if $z_3(t) = 0$, then we have 
    \begin{align*}
        V_1 = \begin{bmatrix}
            0\\ie^{i\theta_1}z_1\\0\\0
        \end{bmatrix},V_2 = \begin{bmatrix}
            0\\0\\ie^{i\theta_2}z_2\\0
        \end{bmatrix}, V_3 = \begin{bmatrix}
            0\\e^{i\theta_1}z'_1\\e^{i\theta_2}z'_2\\0
        \end{bmatrix},
    \end{align*} 
    which shows that $\rank(dF) < 3$ and hence $N$ cannot be a $3$-fold. Thus, we have the equations 
    \begin{align*}
        \im(z_1z_2z_3)' &= 0, \\
        (|z_3|^2-|z_1|^2)' &= 0,\\
        (|z_2|^2-|z_3|^2)' &= 0,
    \end{align*}
    which we know to be a $\mathbb{T}^2$-invariant special Lagrangian $3$-fold from Section \ref{maxtorusSL}.  
    Next, let us assume that $\alpha'_1 \neq 0$. Then, we can eliminate $\alpha'_1$ from the equations. Let us consider the pair of equations \eqref{eqn6} and \eqref{eqn7}. Note that we have \begin{align*}
        (|z_1|^2-|z_2|^2)' \neq 0
    \end{align*}
    because if $(|z_1|^2-|z_2|^2)' = 0$, then as $\alpha'_1\neq 0$, we have $\im(z_1z_2) = \re(z_1 z_2) = 0$ which gives us $z_1 z_2 = 0$. This is a contradiction as $z_1,z_2 \neq 0$ from our assumption above. Hence, eliminating $\alpha'_1$ from \eqref{eqn6} and \eqref{eqn7},
    we have \begin{align*}
    \frac{1}{2}(|z_1|^2-|z_2|^2)' (\re(z_1z_2)\re(z_3)-\im(z_1z_2)\im(z_3)) &= 0\\
    \implies \re(z_1z_2z_3) &= 0.
\end{align*}
Combining this with $\im(z_1z_2z_3)' = 0$, we have that \begin{align*}
    z_1 z_2 z_3 = ki,
\end{align*}
for $k \in \R$. Note that if $k = 0$, then since $z_1, z_2 \neq 0$, we get $z_3 = 0$. But then from \eqref{eqn6} and \eqref{eqn7}, we have $\alpha'_1 = 0$, which is a contradiction. Thus, we can assume that $k \neq 0$. Now, substituting $z_3 = ik/(z_1z_2)$ into the equations \eqref{eqn2}-\eqref{eqn7} gives us
\begin{align*}
    \im\left(\frac{ik}{z_1}\right) \alpha'_1 + \frac{1}{2}\left(|z_2|^2-\frac{k^2}{|z_1z_2|^2}\right)'\re(z_1) &= 0,\\
     \re\left(\frac{ik}{z_1}\right)\alpha'_1 +\frac{1}{2}\left(|z_2|^2-\frac{k^2}{|z_1z_2|^2}\right)'\im(z_1) &= 0,\\
    \im\left(\frac{ik}{z_2}\right) \alpha'_1 +\frac{1}{2}\left(\frac{k^2}{|z_1z_2|^2} -|z_1|^2\right)'\re(z_2) &= 0,\\
    \re\left(\frac{ik}{z_2}\right) \alpha'_1 + \frac{1}{2}\left(\frac{k^2}{|z_1z_2|^2}-|z_1|^2\right)'\im(z_2) &= 0,\\
    \im(z_1z_2) \alpha'_1 + \frac{1}{2}(|z_1|^2-|z_2|^2)'\re\left(\frac{ik}{z_1 z_2}\right) &= 0,\\
     \re(z_1z_2) \alpha'_1 + \frac{1}{2}(|z_1|^2-|z_2|^2)'\im\left(\frac{ik}{z_1 z_2}\right) &= 0.
\end{align*}
Then combining each pair of equations, we get \begin{align}
    \frac{ik}{z_1}\alpha'_1 +\frac{i}{2}\left(|z_2|^2-\frac{k^2}{|z_1z_2|^2}\right)' \overline{z}_1 &= 0 \notag\\
    \implies 2k\alpha'_1 +\left(|z_2|^2-\frac{k^2}{|z_1z_2|^2}\right)' |{z}_1|^2 &= 0, \label{cplx1} \\
    \frac{ik}{z_2}\alpha'_1 +\frac{i}{2}\left(\frac{k^2}{|z_1z_2|^2}-|z_1|^2\right)' \overline{z}_2 &= 0\notag\\
    \implies 2k\alpha'_1 +\left(\frac{k^2}{|z_1z_2|^2}-|z_1|^2\right)' |{z}_2|^2 &= 0, \label{cplx2}\\
    z_1z_2\alpha'_1 +\frac{i}{2}\left(|z_1|^2-|z_2|^2\right)' \overline{\frac{ik}{z_1z_2}} &= 0\notag\\
    \implies 2|z_1z_2|^2\alpha'_1 +k\left(|z_1|^2-|z_2|^2\right)'  &= 0. \label{cplx3}
\end{align}
We want to now show that there are no solutions to the system of equations \eqref{cplx1}-\eqref{cplx3} when $\alpha_1' \neq 0$. For simplification, let us write $x(t) = |z_1(t)|^2, y = |z_2(t)|^2$, and $w(t) = \alpha'_1(t)$. That is, we have the system of equations 
\begin{align}
    2kw+\left(y-\frac{k^2}{xy}\right)'x &= 0, \label{spl1}\\
    2kw+\left(\frac{k^2}{xy}-x\right)'y &= 0 ,\label{spl2}\\
    2xyw+k(x-y)' &= 0, \label{spl3}
\end{align} 
where $x(t), y(t) > 0$, $w(t) \neq 0$
and $k \neq 0$. From equating \eqref{spl1} and \eqref{spl2}, we get \begin{align}
\label{spl4}
    0 = \left(y-\frac{k^2}{xy}\right)'x -\left(\frac{k^2}{xy}-x\right)'y 
    &= \frac{(x^2y^2+k^2x+k^2y)(xy)'}{x^2y^2},
\end{align}
which gives us that $x(t)y(t) = c$ for some $c \in \R_{>0}$ and hence substituting $y(t) = \frac{c}{x(t)}$ in both \eqref{spl1} and \eqref{spl2}, gives us \begin{align}
\label{final1}
    2kw -c\frac{x'}{x} = 0
\end{align}
and \eqref{spl3} gives \begin{align}
\label{final2}
    2cw + k\left(x'+\frac{c}{x^2}x'\right) = 0.
\end{align} 
Substituting \eqref{final1} into \eqref{final2}, we obtain 
\begin{align*}
    \frac{2w}{xc}(c^2 x + k^2 x^2 + k^2 c) = 0 \implies w = 0.
\end{align*}
 Thus, we must have $\alpha'_1 = 0$. Note that we could have taken either $z_1(t),z_3(t) \neq 0$ or $z_2(t),z_3(t) \neq 0$ for all $t$ above instead, which will yield a similar proof.  Therefore, $N$ is always a special Lagrangian in $\C^3$.
\end{proof}

\begin{rmk}
\label{assocremark}
    Note that as we have
    \begin{align*}
        V_1 &= \frac{\pr F}{\pr \theta_1} = ie^{i\theta_1}z_1 \frac{\pr}{\pr z_1}- ie^{-i\theta_1}\overline{z}_1 \frac{\pr}{\pr \overline{z}_1}- ie^{i\theta_3}{z}_3 \frac{\pr}{\pr {z}_3}+ ie^{-i\theta_3}\overline{z}_3 \frac{\pr}{\pr \overline{z}_3},\\
        V_2 &= \frac{\pr F}{\pr \theta_2} = ie^{i\theta_2}z_2 \frac{\pr}{\pr z_2}- ie^{-i\theta_2}\overline{z}_2 \frac{\pr}{\pr \overline{z}_2}- ie^{i\theta_3}{z}_3 \frac{\pr}{\pr {z}_3}+ ie^{-i\theta_3}\overline{z}_3 \frac{\pr}{\pr \overline{z}_3},
    \end{align*}
    we obtain $$\omega(V_1, V_2) = 0.$$ Hence, from \cite[Theorem III.5.5]{harveylawson}, we know that there is a unique special Lagrangian $3$-plane $M$ in $\C^3$ containing the subspace spanned by $V_1$ and $V_2$. As special Lagrangian $3$-planes in $\C^3$ are associatives in $\R^7$ (see Proposition \ref{assocprop}), from uniqueness, it follows that the associative $3$-plane $N$ spanned by $V_1, V_2, V_3$ must be equal to $M$. Theorem \ref{assocslthm} is an explicit demonstration of this fact, without using the existence of the special Lagrangian from \cite[Theorem III.5.5]{harveylawson}.
\end{rmk}

\section{$\bT^3$-invariant Cayley $4$-folds}
\label{cayley}

Similar to the last section, we now consider the action of a maximal torus in $\Spin(7)$ on $\R^8$ and show that a $\mathbb{T}^3$-invariant Cayley $4$-fold must be a special Lagrangian $4$-fold in $\C^4$. Since $\SU(4) \subset \Spin(7)$ and both have rank $3$, any maximal torus for $\SU(4)$ is a maximal torus for $\Spin(7)$. Thus, with respect to the identification $\R^8 \cong \C^4$, we know that the elements $X$ in the Lie algebra of a maximal torus $\bT^3 \subset \Spin(7)$ are of the form \begin{align*}
    X_{\theta_1, \theta_2, \theta_3} &= 
    \begin{bmatrix}
        &-\theta_1\\
        \theta_1\\
        &&&-\theta_2&&\\
        && \theta_2  &&&\\
        &&  &&& -\theta_3\\
        &&&&\theta_3\\
        &&&&&&& -\theta_4\\
        &&&&&&\theta_4 \\
    \end{bmatrix}
\end{align*}
where $\theta_1 + \theta_2 + \theta_3 + \theta_4 = 0$ for this identification. Let us take the complex coordinates
\begin{align*}
    w_1 &= \alpha_0 + i\alpha_1,\\
    w_2 &= \alpha_2 + i\alpha_3, \\
    w_3 &= \alpha_4 + i\alpha_5,\\
    w_4 &= \alpha_6 + i\alpha_7.
\end{align*}
Then, we define the map $F: (-\epsilon, \epsilon) \times  \R^3 \rightarrow \R^8$ as
\begin{align*}
 F(t, \theta_1, \theta_2, \theta_3) = e^{X_{\theta_1, \theta_2, \theta_3}}\cdot \alpha(t)= \begin{bmatrix}
     \alpha_0 \cos\theta_1 -\alpha_1\sin\theta_1\\
     \alpha_0 \sin\theta_1 +\alpha_1\cos\theta_1\\
     \alpha_2 \cos\theta_2 -\alpha_3\sin\theta_2\\
     \alpha_2 \sin\theta_2 +\alpha_3\cos\theta_2\\
     \alpha_4 \cos\theta_3 -\alpha_5\sin\theta_3\\
     \alpha_4 \sin\theta_3 +\alpha_5\cos\theta_3\\
     \alpha_6 \cos\theta_4 -\alpha_7\sin\theta_4\\
     \alpha_6 \sin\theta_4 +\alpha_7\cos\theta_4\\
 \end{bmatrix}
 = \begin{bmatrix}
     e^{i\theta_1}w_1\\
     e^{i\theta_2}w_2\\
     e^{i\theta_3}w_3\\
     e^{i\theta_4}w_4\\
 \end{bmatrix}.
\end{align*}
Thus, the vector fields are given as 
\begin{align*}
    V_1 := \frac{\pr F}{\pr \theta_1} &= i e^{i\theta_1} w_1\frac{\pr}{\pr z_1}- i e^{-i\theta_1} \overline{w_1}\frac{\pr}{\pr \overline{z}_1}- i e^{i\theta_4} w_4\frac{\pr}{\pr z_4}+ i e^{-i\theta_4} \overline{w_4}\frac{\pr}{\pr \overline{z}_4}\\
    V_2 :=  \frac{\pr F}{\pr \theta_2} &= i e^{i\theta_2} w_2\frac{\pr}{\pr z_2}- i e^{-i\theta_2} \overline{w_2}\frac{\pr}{\pr \overline{z}_2}- i e^{i\theta_4} w_4\frac{\pr}{\pr z_4}+i e^{-i\theta_4} \overline{w_4}\frac{\pr}{\pr \overline{z}_4}\\
    V_3 :=    \frac{\pr F}{\pr \theta_3} &= i e^{i\theta_3} w_3\frac{\pr}{\pr z_3}- i e^{-i\theta_3} \overline{w_3}\frac{\pr}{\pr \overline{z}_3}- i e^{i\theta_4} w_4\frac{\pr}{\pr z_4}+ i e^{-i\theta_4} \overline{w_4}\frac{\pr}{\pr \overline{z}_4}\\
    V_4 := \frac{\pr F}{\pr t} &= e^{i\theta_1}\frac{\pr w_1}{\pr t}\frac{\pr}{\pr z_1} + e^{-i\theta_1}\frac{\pr \overline{w_1}}{\pr t}\frac{\pr}{\pr \overline{z}_1}+ e^{i\theta_2}\frac{\pr w_2}{\pr t}\frac{\pr}{\pr z_2} + e^{-i\theta_2}\frac{\pr \overline{w_2}}{\pr t}\frac{\pr}{\pr \overline{z}_2}\\
    &\quad+e^{i\theta_3}\frac{\pr w_3}{\pr t}\frac{\pr}{\pr z_3} + e^{-i\theta_3}\frac{\pr \overline{w_3}}{\pr t}\frac{\pr}{\pr \overline{z}_3}+ e^{i\theta_4}\frac{\pr w_4}{\pr t}\frac{\pr}{\pr z_4} + e^{-i\theta_4}\frac{\pr \overline{w_4}}{\pr t}\frac{\pr}{\pr \overline{z}_4},
\end{align*}
where we used the fact that for functions $f,g$,
\begin{align}
\begin{split}
\label{functions}    
    f'\frac{\pr}{\pr x} + g'\frac{\pr}{\pr y} &= f'\left(\frac{\pr}{\pr z} + \frac{\pr}{\pr \overline{z}}\right) + g'i\left(\frac{\pr}{\pr z} - \frac{\pr}{\pr \overline{z}}\right)\\
    &= (f+ig)' \frac{\pr}{\pr z} + (f-ig)'\frac{\pr}{\pr \overline{z}}.
\end{split}
\end{align}

\begin{theorem}
\label{cayleysl}
      Let $N \subseteq \R^8$ be invariant under the action of the maximal torus $\mathbb{T}^3 \subset \Spin(7)$. Then \begin{align*}
          \text{N is a Cayley $4$-fold in $\R^8$ $\iff$ $N$ is a special Lagrangian in $\C^4$.}
      \end{align*}
\end{theorem}
\begin{proof}
    The Cayley form $\Phi$ can be written as 
    \begin{align*}
    \Phi = \frac{1}{2}\omega \wedge \omega + \re (\Omega), 
    \end{align*}
where 
\begin{align*}
    \omega = \frac{i}{2}\left(dz_1 \wedge d\overline{z}_1+ dz_2 \wedge d\overline{z}_2 + dz_3 \wedge d\overline{z}_3 + dz_4 \wedge d\overline{z}_4\right),
\end{align*}
and 
\begin{align*}
   \Omega =  dz_1 \wedge dz_2 \wedge dz_3 \wedge dz_4.
\end{align*}
Then, 
\begin{align*}
    \omega \wedge \omega &= \frac{-1}{4} (\sum_i dz_i \wedge d\overline{z}_i)\wedge (\sum_j dz_j \wedge d\overline{z}_j)\\
    &= \frac{-1}{4} (\sum_{i \neq j} dz_i \wedge d\overline{z}_i \wedge dz_j \wedge d\overline{z}_j)\\
    &= \frac{1}{2} \sum_{i, j = 1, i < j}^{4}dz_i \wedge dz_j \wedge d\overline{z}_i  \wedge d\overline{z}_j\\
    &= \frac{1}{2}(dz_1 \wedge dz_2 \wedge d\overline{z}_1 \wedge d\overline{z}_2 + dz_1 \wedge dz_3 \wedge d\overline{z}_1 \wedge d\overline{z}_3 + dz_1 \wedge dz_4 \wedge d\overline{z}_1 \wedge d\overline{z}_4 \\
    &\quad+ dz_2 \wedge dz_3 \wedge d\overline{z}_2 \wedge d\overline{z}_3 + dz_2 \wedge dz_4 \wedge d\overline{z}_2 \wedge d\overline{z}_4 + dz_3 \wedge dz_4 \wedge d\overline{z}_3 \wedge d\overline{z}_4).
\end{align*}
A computation gives
\begin{align*}
    (dz_1 \wedge dz_4 \wedge d\overline{z}_1 \wedge d\overline{z}_4)(V_1, V_2, V_3, V_4) &= -dz_1(V_1)dz_4(V_2)d\overline{z}_4(V_3)d\overline{z}_1(V_4)\\
    &\quad+dz_1(V_1)dz_4(V_3)d\overline{z}_4(V_2)d\overline{z}_1(V_4)\\
    &\quad-dz_1(V_4)dz_4(V_3)d\overline{z}_4(V_2)d\overline{z}_1(V_1)\\
    &\quad+dz_1(V_4)dz_4(V_2)d\overline{z}_4(V_3)d\overline{z}_1(V_1)\\
    &= i|w_4|^2\frac{\pr |w_1|^2}{\pr t}-i|w_4|^2\frac{\pr |w_1|^2}{\pr t} = 0
\end{align*}
and similarly, 
\begin{align*}
   (dz_2 \wedge dz_4 \wedge d\overline{z}_2 \wedge d\overline{z}_4)(V_1, V_2, V_3, V_4) = (dz_3 \wedge dz_4 \wedge d\overline{z}_3 \wedge d\overline{z}_4)(V_1, V_2, V_3, V_4) = 0.
\end{align*}
The other three terms in $\omega \wedge \omega$ vanish trivially when evaluated on $V_1, V_2, V_3, V_4$.
Therefore, 
\begin{align*}
    (\omega \wedge \omega)(V_1, V_2, V_3, V_4) = 0.
\end{align*}
We have that $V_1, V_2, V_3, V_4$ span a Cayley subspace iff the subspace spanned by $V_1, V_2, V_3, V_4$ is calibrated by $\Phi$. But as $$\Phi(V_1, V_2, V_3, V_4) = (\re(\Omega))(V_1, V_2, V_3, V_4),$$
$V_1, V_2, V_3, V_4$ span a special Lagrangian $4$-plane in $\C^4$.
\end{proof}

Finally, we show explicitly that the special Lagrangians we obtain in Theorem \ref{cayleysl} are exactly the $\bT^3$-invariant SL $4$-folds we saw in Proposition \ref{tn-1sl}. We have that if $i, j \neq 4$, \begin{align}
\label{LCayley}
    \omega(V_i, V_j) = 0,
\end{align}
and 
\begin{align*}
    \omega(V_1, V_4) &= \frac{-1}{2}(\omega_1\frac{\overline{\omega_1}}{\pr t} + \overline{\omega_1}\frac{{\omega_1}}{\pr t} - (\omega_4\frac{\overline{\omega_4}}{\pr t} + \overline{\omega_4}\frac{{\omega_4}}{\pr t}))\\
    &= \frac{-1}{2}\frac{\pr}{\pr t}(|\omega_1|^2-|\omega_4|^2),\\
    \omega(V_2, V_4) &= \frac{-1}{2}\frac{\pr}{\pr t}(|\omega_2|^2-|\omega_4|^2),\\
    \omega(V_3, V_4) &= \frac{-1}{2 }\frac{\pr}{\pr t}(|\omega_3|^2-|\omega_4|^2).
\end{align*}
So $V_1, V_2, V_3, V_4$ form a Lagrangian submanifold iff \begin{align*}
    |\omega_4|^2-|\omega_k|^2 = c_k,
\end{align*}
for some $c_k \in \R$ and $k = 1, 2, 3$. Furthermore, we have

\begin{align*}
    (dz_1 \wedge dz_2 \wedge dz_3 \wedge dz_4)(V_1, V_2, V_3, V_4) &= dz_1(V_1)dz_2(V_2)dz_3(V_3)dz_4(V_4)\\&\quad-dz_1(V_1)dz_2(V_2)dz_3(V_4)dz_4(V_3)
    \\
    &\quad-dz_1(V_1)dz_2(V_4)dz_3(V_3)dz_4(V_2)\\&\quad-dz_1(V_4)dz_2(V_2)dz_3(V_3)dz_4(V_1)\\
    &= -i\bigg(w_1w_2w_3\frac{\pr w_4}{\pr t}+ w_1w_2\frac{\pr w_3}{\pr t}w_4\\&\quad+ w_1\frac{\pr w_2}{\pr t}w_3w_4+\frac{\pr w_1}{\pr t}w_2w_3w_4\bigg)\\
    &= -i(w_1w_2w_3w_4)'.
\end{align*}

Thus, we get that 
\begin{align*}
    \im(dz_1 \wedge dz_2 \wedge dz_3 \wedge dz_4) = 0 \iff \frac{\pr}{\pr t}\re(w_1 w_2 w_3 w_4) = 0,
\end{align*}
which give us the conditions for  the $4$-plane $N$ spanned by $V_1, V_2, V_3, V_4$ to form a special Lagrangian submanifold.
\begin{rmk}
    Similar to Remark \ref{assocremark}, we have from \eqref{LCayley} and \cite[Theorem III.5.5]{harveylawson} that there is a unique special Lagrangian $4$-plane $M$ in $\C^4$ containing the subspace spanned by $V_1$, $V_2$ and $V_3$. As special Lagrangian $4$-planes in $\C^4$ are Cayleys in $\R^8$ (see Proposition \ref{cayleyprop}), from uniqueness, it follows that the Cayley $4$-plane $N$ spanned by $V_1, V_2, V_3, V_4$ must be equal to $M$. Theorem \ref{cayleysl} is an explicit demonstration of this fact, without using the existence of the special Lagrangian from \cite[Theorem III.5.5]{harveylawson}.
\end{rmk}

\section{$\Sp(1)$-invariant coassociative $4$-folds}
\label{sp1invariant}

We now describe a particular action of the unit quaternions $\Sp(1) \subset \bH$ on $\R^7$ and construct coassociative $4$-folds in $\R^7$ which are invariant under this action. We start off with a more general ansatz that the solution curve must lie in a $4$-plane $\im \bH \oplus e\R^+  \subset \im \bO$, but we show that the curve must lie in a $2$-plane $\R \vec{\epsilon} \oplus e\R^+$ as in  \cite[Section IV.3]{harveylawson}, giving us a rigidity result.  \\

From standard Lie algebra theory, we know that 
\begin{align*}
    T_{\id} S^3 = \im \bH.
\end{align*}
Let $b_1, b_2, b_3$ be an orthonormal basis of $\im \bH$. Then, let $e_1, e_2, e_3$ be the orthonormal frame for $S^3$ given by \begin{align*}
    (e_m)_p = b_m \cdot p, 
\end{align*}
for $p \in S^3$. Note that this is a right-invariant orthonormal frame since if we take $b_m \in T_{\id} S^3 = \im \bH$ and a curve $\gamma: (-\epsilon, \epsilon) \rightarrow S^3$ with $\gamma(0) = \id$ and $\gamma'(0) = b_m$, then for $p \in S^3$,
\begin{align*}
    (R_p)_{*, \id}((e_m)_{\id})= (R_p)_{*, \id}(b_m) &= \frac{d}{dt}\bigg|_{t = 0} (R_p \circ \gamma(t))\\ &= \frac{d}{dt}\bigg|_{t = 0}  \gamma(t) \cdot p= b_m \cdot p = (e_m)_p.
\end{align*}
Let us take the following action of $\Sp(1)$ on $\R^7 \cong \R^3 \oplus \R^4 \cong \im \bH \oplus e\bH$ given as 
\begin{align}
\label{action}
    p \cdot (x, y) = (px\overline{p}, py),
\end{align}
with coordinates $(x_2, x_4, x_6)$ on $\im \bH$ and  $e\bH $ with coordinates $(x_1, x_3, x_5, x_7)$ where $x_i$ are the standard coordinates on $\R^7$. This action embeds $\Sp(1)$ as a subgroup of $\Gt$ (see \cite[Theorem IV.1.8]{harveylawson} for details). Note that the action \eqref{action} differs from the one in \cite{harveylawson} because of our conventions for $\varphi$ and $\psi$ but the actions are equivariantly isomorphic.\\

Now, let us fix a curve  $\alpha: (-\epsilon, \epsilon) \rightarrow  \R^7$  given as \begin{align*}
    \alpha(s) = (\alpha_1(s), \alpha_2(s)).
\end{align*}
Let $F: (-\epsilon, \epsilon) \times S^3 \rightarrow \R^7$ be given as 
\begin{align*}
    F(s, p) = p \cdot \alpha(s) = (p\alpha_1(s)\overline{p}, p\alpha_2(s)). 
\end{align*}

Then, for $m = 1, 2, 3$, we define 
\begin{align*}
    V_m := F_*(e_m) \bigg|_{F(s, p)}&= \frac{d}{dt}\bigg|_{t = 0} F(s, e^{tb_m} \cdot p ) \\
    &= \frac{d}{dt}\bigg|_{t = 0} (e^{tb_m}p\alpha_1(s)\overline{p}e^{-tb_m}, e^{tb_m}p \alpha_2(s) )\\
    &= ([b_m, p\alpha_1(s)\overline{p}], b_m p \alpha_2(s)).
\end{align*}
Let us now choose the basis of $\im \bH$ to be $b_1 = i$, $b_2 = j$, and $b_3 = k$. Then, we get 
\begin{align*}
    V_1 &= ([i, p\alpha_1(s)\overline{p}], i p \alpha_2(s)),\\
    V_2 &= ([j, p\alpha_1(s)\overline{p}], jp \alpha_2(s)),\\
    V_3 &= ([k, p\alpha_1(s)\overline{p}], k p \alpha_2(s)),
\end{align*}
and we define 
\begin{align*}
    V_4 = F_*\left(\frac{\pr}{\pr s}\right) \bigg|_{F(s, p)} &= (p\alpha'_1(s)\overline{p}, p\alpha'_2(s)),
\end{align*}
To compute the $V_i$, let us take 
\begin{align*}
    p &= p_0 + p_1 i + p_2 j + p_3 k,\\
    \alpha_1(s) &= v_1(s) i + v_2(s) j + v_3(s) k ,\\
    \alpha_2(s) &= u_0(s) + u_1(s) i + u_2(s) j + u_3(s) k.
\end{align*}
Then, using quaternionic multiplication, we obtain
\begin{align*}
    V_1 &= [2v_3(-p_0^2+p_1^2+p_2^2-p_3^2)+4v_1(p_0p_2-p_1p_3)-4v_2(p_0p_1+p_3p_2)]\frac{\pr}{\pr x^4}\\&\quad + [2v_2(p_0^2-p_1^2+p_2^2-p_3^2)+4v_1(p_0p_3+p_1p_2)+4v_3(-p_0p_1+p_2 p_3)]\frac{\pr}{\pr x^6}
    \\&\quad-(u_1p_0+u_0p_1+u_3p_2-u_2p_3)\frac{\pr}{\pr x^1}+(u_0p_0-u_1p_1-u_2p_2-u_3p_3)\frac{\pr}{\pr x^3}\\
    &\quad-(u_3p_0+u_2p_1-u_1p_2+u_0p_3)\frac{\pr}{\pr x^5}+ (u_2p_0-u_3p_1+u_0p_2+u_1p_3)\frac{\pr}{\pr x^7},\\
    V_2 &= [2v_3(p_0^2-p_1^2-p_2^2+p_3^2)+4v_1(p_1p_3-p_0p_2) + 4v_2(p_0p_1 + p_2p_3)]\frac{\pr}{\pr x^2}\\
    &\quad+[2v_1(-p_0^2-p_1^2+p_2^2+p_3^2)+4v_2(p_0p_3-p_1p_2) -4v_3(p_0p_2 + p_1p_3)]\frac{\pr}{\pr x^6}\\&\quad-(u_2p_0-u_3p_1+u_0p_2+u_1p_3)\frac{\pr}{\pr x^1}+(u_3p_0+u_2p_1-u_1p_2+u_0p_3)\frac{\pr}{\pr x^3}\\
    &\quad+(u_0p_0-u_1p_1-u_2p_2-u_3p_3)\frac{\pr}{\pr x^5}- (u_1p_0+u_0p_1+u_3p_2-u_2p_3)\frac{\pr}{\pr x^7},\\
    V_3 &= [2v_2(-p_0^2+p_1^2-p_2^2+p_3^2)-4v_1(p_0p_3+p_1p_2) + 4v_3(p_0p_1 - p_2p_3)]\frac{\pr}{\pr x^2}\\
    &\quad+[2v_1(p_0^2+p_1^2-p_2^2-p_3^2)+4v_2(-p_0p_3+p_1p_2) + 4v_3(p_0p_2 + p_1p_3)]\frac{\pr}{\pr x^4}\\&\quad-(u_3p_0+u_2p_1-u_1p_2+u_0p_3)\frac{\pr}{\pr x^1}-(u_2p_0-u_3p_1+u_0p_2+u_1p_3)\frac{\pr}{\pr x^3}\\
    &\quad+(u_1p_0+u_0p_1+u_3p_2-u_2p_3)\frac{\pr}{\pr x^5}+ (u_0p_0-u_1p_1-u_2p_2-u_3p_3)\frac{\pr}{\pr x^7},\\   
    V_4 &= [(p_0^2+p_1^2-p_2^2-p_3^2)v'_1 + 2(p_1p_2-p_0p_3)v'_2 + 2(p_1p_3+p_0p_2)v'_3]\frac{\pr}{\pr x^2}\\
    &\quad+[(p_0^2-p_1^2+p_2^2-p_3^2)v'_2 + 2(p_1p_2+p_0p_3)v'_1-2(p_0p_1-p_2p_3)v'_3]\frac{\pr}{\pr x^4}\\
    &\quad+[ (p_0^2-p_1^2-p_2^2+p_3^2)v'_3 + 2(p_1p_3-p_0p_2)v'_1+2(p_0p_1+p_2p_3)v'_2]\frac{\pr}{\pr x^6}\\
    &\quad+ (u'_0p_0-u'_1p_1-u'_2p_2-u'_3p_3)\frac{\pr}{\pr x^1}+(u'_1p_0+u'_0p_1+u'_3p_2-u'_2p_3)\frac{\pr}{\pr x^3}\\
    &\quad+(u'_2p_0-u'_3p_1+u'_0p_2+u'_1p_3)\frac{\pr}{\pr x^5}+(u'_3p_0+u'_2p_1-u'_1p_2+u'_0p_3)\frac{\pr}{\pr x^7}.
\end{align*}

 We know that when $y \neq 0$ for the action in \eqref{action}, the orbit through $(x, y)$ is diffeomorphic to $S^3$ \cite{harveylawson}. Thus, for every $(x, y) \in \R^7$ with $y \neq 0$, the orbit of $(x, y)$ contains a point of the form $(\wt{x}, t\mathbb{1})$, for $\wt{x} \in \im \bH$ and $t>0$. Indeed, if we take $p = \frac{\overline{y}}{|y|} $, then \begin{align*}
    p \cdot (x, y) = \left(\frac{\overline{y}xy}{|y|^2}, |y|\mathbb{1}\right).
\end{align*}
Hence, we can assume that the curve lies in the plane 
\begin{align*}
    \im \bH \oplus e\R^+  \subset \im \bO.
\end{align*}
 Therefore, we take $u_1(s) = u_2(s) = u_3(s) = 0$. Then, we have the following vector fields: 
 \begin{align}
 \begin{split}
 \label{generalvf}
    V_1 &= [2v_3(-p_0^2+p_1^2+p_2^2-p_3^2)+4v_1(p_0p_2-p_1p_3)-4v_2(p_0p_1+p_3p_2)]\frac{\pr}{\pr x^4}\\&\quad + [2v_2(p_0^2-p_1^2+p_2^2-p_3^2)+4v_1(p_0p_3+p_1p_2)+4v_3(-p_0p_1+p_2 p_3)]\frac{\pr}{\pr x^6}
    \\&\quad+u_0\left(-p_1\frac{\pr}{\pr x^1}+p_0\frac{\pr}{\pr x^3}-p_3\frac{\pr}{\pr x^5}+p_2\frac{\pr}{\pr x^7}\right),\\
    V_2 &= [2v_3(p_0^2-p_1^2-p_2^2+p_3^2)+4v_1(p_1p_3-p_0p_2) + 4v_2(p_0p_1 + p_2p_3)]\frac{\pr}{\pr x^2}\\
    &\quad+[2v_1(-p_0^2-p_1^2+p_2^2+p_3^2)+4v_2(p_0p_3-p_1p_2) -4v_3(p_0p_2 + p_1p_3)]\frac{\pr}{\pr x^6}\\&\quad+u_0\left(-p_2\frac{\pr}{\pr x^1}+p_3\frac{\pr}{\pr x^3}+p_0\frac{\pr}{\pr x^5}- p_1\frac{\pr}{\pr x^7}\right),\\
    V_3 &= [2v_2(-p_0^2+p_1^2-p_2^2+p_3^2)-4v_1(p_0p_3+p_1p_2) + 4v_3(p_0p_1 - p_2p_3)]\frac{\pr}{\pr x^2}\\
    &\quad+[2v_1(p_0^2+p_1^2-p_2^2-p_3^2)+4v_2(-p_0p_3+p_1p_2) + 4v_3(p_0p_2 + p_1p_3)]\frac{\pr}{\pr x^4}\\&\quad+u_0\left(-p_3\frac{\pr}{\pr x^1}-p_2\frac{\pr}{\pr x^3}+p_1\frac{\pr}{\pr x^5}+p_0\frac{\pr}{\pr x^7}\right),\\   
    V_4 &= [(p_0^2+p_1^2-p_2^2-p_3^2)v'_1 + 2(p_1p_2-p_0p_3)v'_2 + 2(p_1p_3+p_0p_2)v'_3]\frac{\pr}{\pr x^2}\\
    &\quad+[(p_0^2-p_1^2+p_2^2-p_3^2)v'_2 + 2(p_1p_2+p_0p_3)v'_1-2(p_0p_1-p_2p_3)v'_3]\frac{\pr}{\pr x^4}\\
    &\quad+[ (p_0^2-p_1^2-p_2^2+p_3^2)v'_3 + 2(p_1p_3-p_0p_2)v'_1+2(p_0p_1+p_2p_3)v'_2]\frac{\pr}{\pr x^6}\\
    &\quad+ u'_0\left(p_0\frac{\pr}{\pr x^1}+p_1\frac{\pr}{\pr x^3}+p_2\frac{\pr}{\pr x^5}+p_3\frac{\pr}{\pr x^7}\right).
\end{split}
\end{align}

We want to now obtain the four coassociative submanifold equations from Proposition \ref{coasociativeprop}: 
\begin{align}
\begin{split}
\label{generaleq}
    \varphi(V_1, V_2, V_3) &= 0,\\
    \varphi(V_1, V_2, V_4) &= 0,\\
    \varphi(V_1, V_3, V_4) &= 0,\\
    \varphi(V_2, V_3, V_4) &= 0.
\end{split}
\end{align}

We will find these equations by entering them into $\varphi$ by direct computation using \eqref{phi}. First, we get
 \begin{align*}
    \varphi(V_1, V_2, V_3) = 0.
\end{align*}
For now, let us fix $p = (1, 0, 0, 0)$ to simplify the equations. Additionally, let us denote $u = u_0$. Then, one computes that

\begin{align}
\label{vec1}
    \varphi(V_2, V_3, V_4) = 0 &\iff 4v_1^2v'_1 +4v_1(-uu'+v_2v'_2+v_3v'_3)-u^2v_1' = 0,\\
\label{vec2}
    \varphi(V_1, V_3, V_4) = 0 &\iff 4v_2^2v'_2 +4v_2(-uu'+v_1v'_1+v_3v'_3)-u^2v_2' = 0,\\
\label{vec3}
    \varphi(V_1, V_2, V_4) = 0 &\iff 4v_3^2v'_3 +4v_3(-uu'+v_1v'_1+v_2v'_2)-u^2v_3' = 0.
\end{align}
We observe that the equations \eqref{vec1}-\eqref{vec3} can be written as the following single vector equation 
\begin{align}
\label{veceq}
   2 \frac{d|\vec{v}|^2}{dt}\vec{v}-2\frac{du^2}{dt}\vec{v}-u^2\frac{d\vec{v}}{dt} = 0,
\end{align}
where $\vec{v} = (v_1, v_2, v_3)$. Taking the dot product of \eqref{veceq} with $2\vec{v}$, as $2 \vec{v}\cdot\frac{d\vec{v}}{dt} = \frac{d|\vec{v}|^2}{dt}$, we obtain the following ODE for $|\vec{v}|^2$, 
\begin{align*}
    4\frac{d|\vec{v}|^2}{dt}|\vec{v}|^2-4\frac{du^2}{dt}|\vec{v}|^2-u^2 \frac{d|\vec{v}|^2}{dt} = 0.
\end{align*}

Denote $F(t) = |\vec{v}(t)|^2$ and $G(t) = u(t)$. Then the equation \begin{align}
\label{FGeq}
    4F\frac{dF}{dt}-8FG\frac{dG}{dt}-G^2\frac{dF}{dt} = 0,
\end{align}
has the solution \begin{align}
\label{FGsol}
    F(5G^2-4F)^4 = k,
\end{align}
for some constant $k \in \R$. From this, we get the following expression for $G$:
\begin{align*}
    G^2(t) = \frac{4F(t)\pm\left(\frac{k}{F(t)}\right)^{1/4}}{5}.
\end{align*}
If $k= 0$, then we have  
\begin{align}
\label{fg2}
    F(t) = \frac{5G^2(t)}{4}.
\end{align}
Substituting this in the original equation \eqref{veceq}, we have for each $v_i$, \begin{align*}
    5G\frac{dG}{dt}v_i-4G\frac{dG}{dt}v_i-G^2\frac{dv_i}{dt} = G \frac{dG}{dt}v_i-G^2\frac{dv_i}{dt} = 0.
\end{align*}
 Therefore, we have the solution
\begin{align*}
    v_i = c_i G,
\end{align*}
for constants $c_i \in \R$. Thus, when $k = 0$ our solution is of the form 
\begin{align}
\label{c1c2c3sol}
    (c_1 u(t), c_2u(t), c_3u(t), u(t), 0, 0, 0),
\end{align}
for constants $c_1, c_2, c_3 \in \R$. That is, the solutions are given as 
\begin{align}
\label{cone}
    \alpha(t) = u(t)\cdot(\vec{c}, e),
\end{align}
where $\vec{c} = (c_1, c_2, c_3) \in \R^3$. This means that the curve $\alpha$ lies in the plane \begin{align*}
    \R \vec{c} \oplus e\R^+ \subset \im \bO.
\end{align*} 
Furthermore, from \eqref{fg2}, we have that \begin{align*}
    c_1^2 + c_2^2+c_3^2 = \frac{5}{4}.
\end{align*}
Note that the solutions \eqref{cone} form the cone 
\begin{align*}
    M_0 := \left\{\lambda(\vec{c}, e): \lambda \in \R_{\geq 0},\  \vec{c} \in \R^3,\ c_1^2 + c_2^2+c_3^2 = \frac{5}{4} \right\}.
\end{align*}  

Let us now consider the case when $k \neq 0$. First, let us assume arc-length parametrization. That is, 
\begin{align}
\label{arclength}
    \left(\frac{dv_1}{dt}\right)^2 + \left(\frac{dv_2}{dt}\right)^2+ \left(\frac{dv_3}{dt}\right)^2+ \left(\frac{du}{dt}\right)^2 = \left|\frac{d\vec{v}}{dt}\right|^2 + \left|\frac{du}{dt}\right|^2 = 1.
\end{align}

Then, taking the dot product of \eqref{veceq} with $2\frac{d\vec{v}}{dt}$ and substituting \eqref{arclength} in the expression, we have 
\begin{align*}
     0 &= 2 \left(\frac{d|\vec{v}|^2}{dt}\right)^2-4u\frac{du}{dt}\frac{d|\vec{v}|^2}{dt}-2u^2\left|\frac{d\vec{v}}{dt}\right|^2\\
     &= 2\left(\frac{d|\vec{v}|^2}{dt}\right)^2-4u\frac{du}{dt}\frac{d|\vec{v}|^2}{dt}-2u^2\left(1-\left|\frac{d{u}}{dt}\right|^2\right)
\end{align*}
Taking $F = |\vec{v}|^2$ and $G = u$ again, we have 
\begin{align*}
    2\left(\frac{dF}{dt}\right)^2-4G\frac{dG}{dt}\frac{dF}{dt}-2G^2\left(1-\left(\frac{dG}{dt}\right)^2\right),
\end{align*}
from which we obtain \begin{align}
\label{dfdt1}
    \frac{dF}{dt} = G\frac{dG}{dt}\pm G.
\end{align}
Substituting the above in \eqref{FGeq} we have 
\begin{align*}
    4F\left(G\frac{dG}{dt}\pm G\right)-8FG\frac{dG}{dt}-G^2\left(G\frac{dG}{dt}\pm G\right) = 0.
\end{align*}
It follows that
\begin{align*}
    4F\left(\frac{dG}{dt}\pm 1\right)-8F\frac{dG}{dt}-G\left(G\frac{dG}{dt}\pm G\right) = (-4F-G^2)\frac{dG}{dt}\pm (4F-G^2)= 0,
\end{align*}
which gives us
\begin{align}
\label{dgdt1}
    \frac{dG}{dt} = \pm \frac{4F-G^2}{4F+G^2},
\end{align}
and 
\begin{align}
\label{dfdt2}
    \frac{dF}{dt} = \pm G  \frac{4F-G^2}{4F+G^2}\pm G = \pm \frac{8FG}{4F+G^2}.
\end{align}
We can omit the $\pm$ above since it corresponds to $t \rightarrow -t$. Substituting \eqref{dfdt2} and \eqref{dgdt1} in \eqref{veceq}, we get \begin{align}
\label{dvdt}
      \frac{16FG}{4F+G^2}\vec{v}-4G\left(\frac{4F-G^2}{4F+G^2}\right)\vec{v}-G^2\frac{d\vec{v}}{dt} = 0 \implies \frac{d\vec{v}}{dt} = \frac{4G}{4F+G^2}\vec{v},
\end{align}
as $G \neq 0$. Denoting $H(t) = \frac{4G(t)}{4F(t)+G(t)^2}$, we obtain 
\begin{align*}
     \frac{dv_i}{v_i} = H(t) dt \implies \int H(t) dt = \log v_i + b_i,
\end{align*}
for constants $b_i \in \R$, where $i = 1, 2, 3$. That is, \begin{align}
\label{vi}
    v_i = \epsilon_i e^{\int H(t) dt},
\end{align} 
for some constants $\epsilon_i \in \R$. Therefore, we again have that the curve $\alpha$ lies in the plane
\begin{align*}
    \R \vec{\epsilon} \oplus e\R^+ \subset \im \bO,
\end{align*}
where $\vec{\epsilon} = (\epsilon_1, \epsilon_2, \epsilon_3) \in \R^3$ is a constant vector.
\\

Substituting the solutions for $k = 0$ given by \eqref{c1c2c3sol} and for $k \neq 0$ given by \eqref{dgdt1}, \eqref{dfdt2}, \eqref{dvdt} in the equations obtained from \eqref{generalvf} and \eqref{generaleq}, we find that they satisfy the original equations for general points $p \in S^3$ and not just $p = (1, 0, 0, 0)$. Therefore, we have shown the following: 

\begin{theorem}
\label{sp1rigidity}
    Let  $\alpha(t) = (\vec{v}(t), u(t), 0, 0, 0)$ be the curve in $\im \bH \oplus e\R^+  \subset \im \bO$ and $V_1, V_2, V_3, V_4$ be linearly indpendent vector fields defined as above. In addition, assume that $\alpha$ is parametrized by arclength. The $V_i$ span a coassociative subspace if and only if 
    \begin{enumerate}
        \item The curve $\alpha(t)$ lies in the plane 
        \begin{align*}
        \R \vec{\epsilon} \oplus e\R^+ \subset \im \bO,
        \end{align*}
        where $\vec{\epsilon} \in \R^3$ is a constant vector, and
        \item For $F = |\vec{v}|^2$, $G = u$ and some constant $k \in \R$,, we have 
        $$F(5G^2-4F)^4 = k.$$      
    \end{enumerate}
    In particular, if $k = 0$, then we have 
    \begin{align*}
        \alpha(t) = u(t)\cdot(\vec{c}, e),
    \end{align*}
    where $\vec{c} = (c_1, c_2, c_3) \in \R^3$ is a constant vector such that 
    \begin{align*}
    c_1^2 + c_2^2+c_3^2 = \frac{5}{4}.
    \end{align*}
\end{theorem}

Note that in \cite[Section IV.3]{harveylawson} they make the assumption that the curve lies in a $2$-plane $\R \vec{\epsilon} \oplus e\R^+ \subset \im \bO$ for some constant vector $\vec{\epsilon}$. Even though we started off with the more general assumption that the curve lies in a  $4$-plane $\im \bH \oplus e\R^+  \subset \im \bO$, from \eqref{cone} and \eqref{vi} we saw that the curve must lie in a $2$-plane $\R \vec{\epsilon} \oplus e\R^+ $.

\section{$\mathbb{T}^2$-invariant coassociative $4$-folds}
\label{t2invariantcoass}
We construct examples of cohomogeneity two $\bT^2$-invariant coassociative $4$-folds in this section. Let us consider the $\mathbb{T}^2$-action from Section \ref{t2invariantassoc}. We let $\alpha: (-\epsilon_1, \epsilon_1) \times (-\epsilon_2, \epsilon_2) \rightarrow \R^7$  be  a surface in $\R^7$. Then, with respect to the complex coordinates 
\begin{align*}
    w_1(s, t) &= \alpha_2(s, t) + i \alpha_3(s, t),\\
    w_2(s, t) &= \alpha_4(s, t) + i\alpha_5(s, t), \\
    w_3(s, t) &= \alpha_6(s, t) + i \alpha_7(s, t),
\end{align*}
we have from \eqref{maxtorusmap}, that the map is given as
\begin{align*}
    F(\theta_1, \theta_2,s, t) = \begin{bmatrix}
        \alpha_1(s,t)\\ e^{i\theta_1}\cdot w_1(s,t)\\
        e^{i\theta_2}\cdot w_2(s, t)\\
        e^{i\theta_3}\cdot w_3(s,t)
    \end{bmatrix},
\end{align*}
where $\theta_1+ \theta_2 + \theta_3 = 0$. As in Section \ref{t2invariantassoc}, we have that the orbit through a point is $2$-dimensional if and only if at most one of $w_1, w_2, w_3$ is zero. Using \eqref{functions}, we get the vector fields
\begin{align*}
    V_1 := \frac{\pr F}{\pr \theta_1} &= ie^{i\theta_1}w_1 \frac{\pr}{\pr z_1}- ie^{-i\theta_1}\overline{w}_1 \frac{\pr}{\pr \overline{z}_1}- ie^{i\theta_3}{w}_3 \frac{\pr}{\pr {z}_3}+ ie^{-i\theta_3}\overline{w}_3 \frac{\pr}{\pr \overline{z}_3}\\
     V_2 := \frac{\pr F}{\pr \theta_2} &= ie^{i\theta_2}w_2 \frac{\pr}{\pr z_2}- ie^{-i\theta_2}\overline{w}_2\frac{\pr}{\pr \overline{z}_2}- ie^{i\theta_3}{w}_3 \frac{\pr}{\pr {z}_3}+ ie^{-i\theta_3}\overline{w}_3 \frac{\pr}{\pr \overline{z}_3}\\
     V_3 := \frac{\pr F}{\pr s} &= \frac{\pr \alpha_1}{\pr s}\frac{\pr}{\pr x_1} + e^{i\theta_1}\frac{\pr w_1}{\pr s} \frac{\pr}{\pr z_1}+e^{-i\theta_1}\frac{\pr \overline{w_1}}{\pr s} \frac{\pr}{\pr \overline{z}_1}+e^{i\theta_2}\frac{\pr w_2}{\pr s} \frac{\pr}{\pr z_2}+e^{-i\theta_2}\frac{\pr \overline{w_2}}{\pr {s}} \frac{\pr}{\pr \overline{z}_2}\\
     &\quad+e^{i\theta_3}\frac{\pr w_3}{\pr s} \frac{\pr}{\pr z_3}+e^{-i\theta_3}\frac{\pr \overline{w_3}}{\pr {s}} \frac{\pr}{\pr \overline{z}_3}\\
     V_4 :=  \frac{\pr F}{\pr t} &= \frac{\pr \alpha_1}{\pr t}\frac{\pr}{\pr x_1} + e^{i\theta_1}\frac{\pr w_1}{\pr t} \frac{\pr}{\pr z_1}+e^{-i\theta_1}\frac{\pr \overline{w_1}}{\pr t} \frac{\pr}{\pr \overline{z}_1}+e^{i\theta_2}\frac{\pr w_2}{\pr t} \frac{\pr}{\pr z_2}+e^{-i\theta_2}\frac{\pr \overline{w_2}}{\pr {t}} \frac{\pr}{\pr \overline{z}_2}\\
       &\quad+e^{i\theta_3}\frac{\pr w_3}{\pr t} \frac{\pr}{\pr z_3}+e^{-i\theta_3}\frac{\pr \overline{w_3}}{\pr {t}} \frac{\pr}{\pr \overline{z}_3}.
\end{align*}

We can write the $\Gt$-form $\varphi$ as 
\begin{align*}
    \varphi = dx_1 \wedge \omega + \re(dz_1 \wedge dz_2 \wedge dz_3), 
\end{align*}
where 
\begin{align*}
    \omega = \frac{i}{2}\left(dz_1 \wedge d\overline{z}_1+ dz_2 \wedge d\overline{z}_2 + dz_3 \wedge d\overline{z}_3 \right).
\end{align*}
Let us first look at the equation 
$$\varphi(V_1, V_2, V_3) = 0.$$
We have 
\begin{align*}
   ( dx_1 \wedge \omega )(V_1, V_2, V_3) &= dx_1(V_3)\omega(V_1, V_2)-dx_1(V_3)\omega(V_2, V_1)\\
   &= 2\frac{\pr \alpha_1}{\pr s}\frac{i}{2}\left(dz_3(V_1)d\overline{z}_3(V_2)-dz_3(V_2)d\overline{z}_3(V_1)\right) = 0
\end{align*}
and 
\begin{align*}
    (\re(dz_1 \wedge dz_2 \wedge dz_3))(V_1, V_2, V_3) &= \re(dz_1(V_1)dz_2(V_2)dz_3(V_3)-dz_1(V_3)dz_2(V_2)dz_3(V_1)\\&\quad-dz_1(V_1)dz_2(V_3)dz_3(V_2))\\
    &= \frac{1}{2}(dz_1(V_1)dz_2(V_2)dz_3(V_3)-dz_1(V_3)dz_2(V_2)dz_3(V_1)\\&\quad-dz_1(V_1)dz_2(V_3)dz_3(V_2)+d\overline{z}_1(V_1)d\overline{z}_2(V_2)d\overline{z}_3(V_3)\\&\quad-d\overline{z}_1(V_3)d\overline{z}_2(V_2)d\overline{z}_3(V_1)-d\overline{z}_1(V_1)d\overline{z}_2(V_3)d\overline{z}_3(V_2))\\
    &= -\frac{1}{2}\bigg(w_1 w_2 \frac{\pr w_3}{\pr s}+ \frac{\pr w_1}{\pr s}w_2 w_3 + w_1 \frac{\pr w_2}{\pr s}w_3 \\&\quad+\overline{w_1 w_2} \frac{\pr \overline{w_3}}{\pr s}+ \frac{\pr \overline{w_1}}{\pr s}\overline{w_2 w_3} + \overline{w_1} \frac{\pr \overline{w_2}}{\pr s}\overline{w_3} \bigg)\\
    &= -\frac{1}{2}\left(\frac{\pr(w_1w_2w_3)}{\pr s}+\frac{\pr(\overline{w_1w_2w_3})}{\pr s}\right)\\
    &= -\frac{\pr(\re(w_1w_2w_3))}{\pr s}.
\end{align*}
Thus, 
\begin{align*}
    \varphi(V_1, V_2, V_3) = 0 \iff \frac{\pr}{\pr s}\left(\re(w_1w_2w_3)\right)= 0.
\end{align*}
Similarly, 
\begin{align*}
    \varphi(V_1, V_2, V_4) = 0 \iff \frac{\pr}{\pr t}\left(\re(w_1w_2w_3)\right)= 0.
\end{align*}
Thus, $\re(w_1 w_2 w_3)$ must be constant. Now, consider the equation
$$\varphi(V_1, V_3, V_4) = 0.$$
We have 
\begin{align*}
    (dx_1 \wedge \omega)(V_1, V_3, V_4) &= -dx_1(V_3) \omega(V_1, V_4) +dx_1(V_3) \omega(V_4, V_1)\\
    &\quad-dx_1(V_4) \omega(V_3, V_1) +dx_1(V_4) \omega(V_1, V_3)\\
    &= -2\frac{\pr \alpha_1}{\pr s}\omega(V_1, V_4) + 2\frac{\pr \alpha_1}{\pr t}\omega(V_1, V_3) .
\end{align*}
Then, 
\begin{align*}
    \omega(V_1, V_4) &= \frac{i}{2}\left(dz_1(V_1)d\overline{z}_1(V_4)-dz_1(V_4)d\overline{z}_1(V_1)+dz_3(V_1)d\overline{z}_3(V_4)-dz_3(V_4)d\overline{z}_3(V_1)\right)\\
    &= \frac{-1}{2}\left(w_1 \frac{\pr \overline{w_1}}{\pr t}+\overline{w}_1 \frac{\pr w_1}{\pr t}-w_3 \frac{\pr \overline{w_3}}{\pr t}-\overline{w}_3 \frac{\pr w_3}{\pr t}\right),
\end{align*}
and 
\begin{align*}
    \omega(V_1, V_3) &= \frac{-1}{2}\left(w_1 \frac{\pr \overline{w_1}}{\pr s}+\overline{w}_1 \frac{\pr w_1}{\pr s}-w_3 \frac{\pr \overline{w_3}}{\pr s}-\overline{w}_3 \frac{\pr w_3}{\pr s}\right).
\end{align*}
So we have 
\begin{align*}
    (dx_1 \wedge \omega)(V_1, V_3, V_4) &= \frac{\pr \alpha_1}{\pr s}\left(w_1 \frac{\pr \overline{w_1}}{\pr t}+\overline{w}_1 \frac{\pr w_1}{\pr t}-w_3 \frac{\pr \overline{w_3}}{\pr t}-\overline{w}_3 \frac{\pr w_3}{\pr t}\right) \\&\quad- \frac{\pr \alpha_1}{\pr t}\left(w_1 \frac{\pr \overline{w_1}}{\pr s}+\overline{w}_1 \frac{\pr w_1}{\pr s}-w_3 \frac{\pr \overline{w_3}}{\pr s}-\overline{w}_3 \frac{\pr w_3}{\pr s}\right),
\end{align*}
and 
\begin{align*}
   (\re(dz_1 \wedge dz_2 \wedge dz_3))(V_1, V_3, V_4) &= \re(dz_1(V_1)dz_2(V_3)dz_3(V_4)- dz_1(V_1)dz_2(V_4)dz_3(V_3)\\
   &\quad-dz_1(V_4)dz_2(V_3)dz_3(V_1)+dz_1(V_3)dz_2(V_4)dz_3(V_1))\\
   &= \re\bigg(i\left(w_1\frac{\pr w_2}{\pr s}\frac{\pr w_3}{\pr t} - w_1\frac{\pr w_2}{\pr t}\frac{\pr w_3}{\pr s}+w_3 \frac{\pr w_1}{\pr t}\frac{\pr w_2}{\pr s}-w_3 \frac{\pr w_1}{\pr s}\frac{\pr w_2}{\pr t}\right)\bigg)\\
   &= -\im\left(w_1 \left(\frac{\pr w_2}{\pr s}\frac{\pr w_3}{\pr t}-\frac{\pr w_2}{\pr t}\frac{\pr w_3}{\pr s}\right) + w_3\left(\frac{\pr w_1}{\pr t}\frac{\pr w_2}{\pr s}-\frac{\pr w_1}{\pr s}\frac{\pr w_2}{\pr t}\right)\right).
\end{align*}
Therefore, 
\begin{align*}
    \varphi(V_1, V_3, V_4) = 0 &\iff \frac{\pr \alpha_1}{\pr s}\left(w_1 \frac{\pr \overline{w_1}}{\pr t}+\overline{w}_1 \frac{\pr w_1}{\pr t}-w_3 \frac{\pr \overline{w_3}}{\pr t}-\overline{w}_3 \frac{\pr w_3}{\pr t}\right) \\&\quad- \frac{\pr \alpha_1}{\pr t}\left(w_1 \frac{\pr \overline{w_1}}{\pr s}+\overline{w}_1 \frac{\pr w_1}{\pr s}-w_3 \frac{\pr \overline{w_3}}{\pr s}-\overline{w}_3 \frac{\pr w_3}{\pr s}\right)\\
    &\quad-\im\left(w_1 \left(\frac{\pr w_2}{\pr s}\frac{\pr w_3}{\pr t}-\frac{\pr w_2}{\pr t}\frac{\pr w_3}{\pr s}\right) + w_3\left(\frac{\pr w_1}{\pr t}\frac{\pr w_2}{\pr s}-\frac{\pr w_1}{\pr s}\frac{\pr w_2}{\pr t}\right)\right) = 0\\
    &\iff \frac{\pr \alpha_1}{\pr s}\left( \frac{\pr |{w_1}|^2}{\pr t}- \frac{\pr |w_3|^2}{\pr t}\right) - \frac{\pr \alpha_1}{\pr t}\left( \frac{\pr |{w_1}|^2}{\pr s}- \frac{\pr |w_3|^2}{\pr s} \right)\\
    &\quad-\im\left(w_1 \left(\frac{\pr w_2}{\pr s}\frac{\pr w_3}{\pr t}-\frac{\pr w_2}{\pr t}\frac{\pr w_3}{\pr s}\right) + w_3\left(\frac{\pr w_1}{\pr t}\frac{\pr w_2}{\pr s}-\frac{\pr w_1}{\pr s}\frac{\pr w_2}{\pr t}\right)\right) = 0.
\end{align*}
And similarly, 
\begin{align*}
     \varphi(V_2, V_3, V_4) = 0 &\iff \frac{\pr \alpha_1}{\pr s}\left( \frac{\pr |{w_2}|^2}{\pr t}- \frac{\pr |w_3|^2}{\pr t}\right) - \frac{\pr \alpha_1}{\pr t}\left( \frac{\pr |{w_2}|^2}{\pr s}- \frac{\pr |w_3|^2}{\pr s} \right)\\
    &\quad-\im\left(w_2 \left(\frac{\pr w_1}{\pr t}\frac{\pr w_3}{\pr s}-\frac{\pr w_1}{\pr s}\frac{\pr w_3}{\pr t}\right) + w_3\left(\frac{\pr w_1}{\pr t}\frac{\pr w_2}{\pr s}-\frac{\pr w_1}{\pr s}\frac{\pr w_2}{\pr t}\right)\right) = 0.
\end{align*}
Finally, note that to obtain a coassociative $4$-fold, we need $F$ to be a immersion. That is, the vector fields $V_1, V_2, V_3, V_4$ need to be linearly independent. To summarize, we have shown the following: 
\begin{theorem}
    Let $V_1, V_2, V_3, V_4$ be linearly independent vector fields as above, $\alpha(s, t)$ be the surface and $w_1, w_2, w_3$ be the complex coordinates as defined above. The $V_i$ span a coassociative subspace if and only if 
    \begin{align}
    \label{coasseq1}
        \re(w_1w_2w_3) = c,
    \end{align}
    for some $c \in \R$, 
    \begin{align} 
    \begin{split} 
    \label{coasseq2}
        &\frac{\pr \alpha_1}{\pr s}\left( \frac{\pr |{w_1}|^2}{\pr t}- \frac{\pr |w_3|^2}{\pr t}\right) - \frac{\pr \alpha_1}{\pr t}\left(\frac{\pr |{w_1}|^2}{\pr s}- \frac{\pr |w_3|^2}{\pr s} \right)\\
    &\quad-\im\left(w_1 \left(\frac{\pr w_2}{\pr s}\frac{\pr w_3}{\pr t}-\frac{\pr w_2}{\pr t}\frac{\pr w_3}{\pr s}\right) + w_3\left(\frac{\pr w_1}{\pr t}\frac{\pr w_2}{\pr s}-\frac{\pr w_1}{\pr s}\frac{\pr w_2}{\pr t}\right)\right) = 0,
    \end{split}
    \end{align}
    and 
    \begin{align}
    \label{coasseq3}
        \begin{split}
            &\frac{\pr \alpha_1}{\pr s}\left( \frac{\pr |{w_2}|^2}{\pr t}- \frac{\pr |w_3|^2}{\pr t}\right) - \frac{\pr \alpha_1}{\pr t}\left( \frac{\pr |{w_2}|^2}{\pr s}- \frac{\pr |w_3|^2}{\pr s} \right)\\
    &\quad-\im\left(w_2 \left(\frac{\pr w_1}{\pr t}\frac{\pr w_3}{\pr s}-\frac{\pr w_1}{\pr s}\frac{\pr w_3}{\pr t}\right) + w_3\left(\frac{\pr w_1}{\pr t}\frac{\pr w_2}{\pr s}-\frac{\pr w_1}{\pr s}\frac{\pr w_2}{\pr t}\right)\right) = 0.
        \end{split}
    \end{align}
\end{theorem}

Now, we give explicit examples of $\mathbb{T}^2$-invariant coassociatives. 
\begin{example}
\label{example0}
    Let us take 
    \begin{align*}
        \alpha_1 = s+t^l,\quad w_1 = f(t),\quad w_2 = f(t), \quad w_3 = if(t),
    \end{align*}
    for a real-valued function $f$ such that $f(t) \neq 0$ and $f'(t) \neq 0$ for all $t$ and some $l \in \Z_{>0}$. First, note that $V_1, V_2, V_3, V_4$ are linearly independent. Indeed, we have 
    \begin{align*}
        V_1 = \begin{bmatrix}
            0\\ie^{i\theta_1}f\\0\\e^{i\theta_3}f
        \end{bmatrix},V_2 = \begin{bmatrix}
            0\\0\\ie^{i\theta_2}f\\e^{i\theta_3}f
        \end{bmatrix}, V_3 = \begin{bmatrix}
            1\\0\\0\\0
        \end{bmatrix},V_4 = \begin{bmatrix}
            lt^{l-1}\\e^{i\theta_1}f'\\e^{i\theta_2}f'\\ie^{i\theta_3}f'
        \end{bmatrix},
    \end{align*}
    and if $a_1V_1 + a_2 V_2 + a_3 V_3 + a_4 V_4 = 0$ for $a_i \in \R$, then we have 
    \begin{align*}
        a_3 +a_4lt^{l-1}  = 0, \quad -a_1f + a_4 f'i = 0, \quad -a_2f + a_4 f'i = 0, \quad a_1 f + a_2 f + a_4 f'i = 0,
    \end{align*}
    from which we obtain $a_1 = a_2 = a_3 = a_4 = 0$. Then, as $w_1 w_2 w_3 = if^3, \frac{\pr f}{\pr s} = 0$ and $|w_1|^2=|w_2|^2=|w_3|^2$, we have that \eqref{coasseq1}, \eqref{coasseq2} and \eqref{coasseq3} are satisfied. 
\end{example}

\begin{example}
\label{example1}
    Take 
    \begin{align*}
        \alpha_1 = s+t,\quad w_1 = f(s)g(t),\quad w_2 = f(s)g(t), \quad w_3 = if(s)g(t),
    \end{align*}
    for real-valued functions $f,g$ such that $f(s),g(t) \neq 0$ and $f(s)g'(t)-f'(s)g(t) \neq 0$ for all $s,t$ (for instance, take $f(s) = e^{2s}, g(t) = e^t$). Then, similar to Example \ref{example0}, since $w_1 w_2 w_3 = if^3g^3$ and $|w_1|^2=|w_2|^2=|w_3|^2$, we have that \eqref{coasseq1}, \eqref{coasseq2} and \eqref{coasseq3} are satisfied. Furthermore, $V_1, V_2, V_3, V_4$ are linearly independent as 
    \begin{align*}
        V_1 = \begin{bmatrix}
            0\\ie^{i\theta_1}fg\\0\\e^{i\theta_3}fg
        \end{bmatrix},V_2 = \begin{bmatrix}
            0\\0\\ie^{i\theta_2}fg\\e^{i\theta_3}fg
        \end{bmatrix}, V_3 = \begin{bmatrix}
            1\\e^{i\theta_1}f'g\\e^{i\theta_2}f'g\\ie^{i\theta_3}f'g
        \end{bmatrix},V_4 = \begin{bmatrix}
            1\\e^{i\theta_1}fg'\\e^{i\theta_2}fg'\\ie^{i\theta_3}fg'
        \end{bmatrix},
    \end{align*}
    and if $a_1 V_1 + a_2 V_2 + a_3 V_3 + a_4 V_4 = 0$ for $a_i \in \R$, then 
    \begin{gather*}
        a_3 + a_4 = 0, \quad a_1ifg + a_3f'g+a_4fg' = 0, \quad a_2ifg + a_3f'g+a_4fg' = 0,\\ a_1 fg + a_2 fg + ia_3 f'g + ia_4 fg' = 0. 
    \end{gather*}
    Since $f, g \neq 0$ and $f(s)g'(t)-f'(s)g(t) \neq 0$ , we have that $a_1 = a_2 = a_3 = a_4 = 0$.
\end{example}

\begin{rmk}
    Taking $\alpha_1 = s^m +t^l$ for some $m \in \Z_{>0}$ instead of $\alpha_1 = s +t^l$ in Example \ref{example0} gives us a coassociative $4$-fold except for points where $s = 0$. Furthermore, we can generalize the above examples by taking $w_1, w_2, w_3$ such that $\re(w_1 w_2 w_3)$ is constant and then solving the equations \eqref{coasseq2} and \eqref{coasseq3} which give linear PDEs for $\frac{\pr \alpha_1}{\pr s}$ and $\frac{\pr \alpha_1}{\pr t}$.
\end{rmk}
The following example illustrates the importance of $V_1, V_2, V_3, V_4$ being linearly independent.
\begin{example}
\label{singularexample}
    Let us take the surface 
\begin{align*}
    \alpha_1 = s^2 + t^2,\quad w_1 = s + it,\quad w_2 = c(s+it),\quad w_3 = i(s-it)^2,
\end{align*}
for some $c \in \R$.\\

Note that \begin{align*}
    w_1w_2w_3 = c(s+it)(s+it)i(s-it)^2 = ci(s^2+t^2)^2 ,
\end{align*}
and thus 
\begin{align*}
    \re(w_1w_2w_3) = 0,
\end{align*}
which means the equation \eqref{coasseq1} is satisfied.
Then, substituting these expressions in equation \eqref{coasseq2} gives us 
\begin{align*}
    &2s(2t-2(s^2+t^2)2t)
    - 2t(2s-2(s^2+t^2)2s) \\&\quad-\im((s+it)(2c(s-it) +2c(s-it)) + i(s-it)^2(ci - ci)) \\
    &= -\im(4c(s+it)(s-it)) = -4c\im(s^2+t^2) = 0,
\end{align*}
and doing the same for \eqref{coasseq3} gives us 
\begin{align*}
    &2s(2c^2t-2(s^2+t^2)2t) -2t(2c^2s-2(s^2+t^2)2s)\\
    &\quad-\im(c(s+it)(-2(s-it)-2(s-it)) + i(s-it)^2(ci - ci)) \\
    &= \im(4c(s+it)(s-it)) = 4c\im(s^2+t^2) = 0.   
\end{align*}
But this cannot describe a coassociative $4$-fold since we have that $V_1, V_2, V_3, V_4$ are linearly dependent as one can show that $V_1 + V_2 + tV_3 -sV_4 = 0$. 
\end{example}

Let us now give a simpler recipe to construct these coassociatives from real-valued functions. First, notice that we can rotate $w_1, w_2$ to make them real-valued. That is, we can always choose angles $\theta_1, \theta_2$ such that $e^{i\theta_1}\cdot w_1(s, t)$ and $e^{i\theta_2}\cdot w_2(s, t)$ are real. So let us assume that 
\begin{align*}
    w_1(s, t) = \alpha_2(s, t), \\
    w_2(s, t) = \alpha_4(s, t). 
\end{align*}
Then, from \eqref{coasseq1}, we have that for some $c \in \R$, \begin{align*}
    \alpha_2 \alpha_4 \alpha_6 = c \implies \alpha_6 = \frac{c}{\alpha_2 \alpha_4}.
\end{align*}
Using this, \eqref{coasseq2} gives us 
\begin{align}
\begin{split}
\label{smpl1}   
    &\frac{\pr \alpha_1}{\pr s}\left(\frac{\pr\left(\alpha_2^2-\left(\frac{c^2}{(\alpha_2\alpha_4)^2 }+ \alpha_7^2\right)\right)}{\pr t}\right) -\frac{\pr \alpha_1}{\pr t}\left( \frac{\pr\left(\alpha_2^2-\left(\frac{c^2}{(\alpha_2\alpha_4)^2 }+ \alpha_7^2\right)\right)}{\pr s}  \right)\\
    &\quad+\frac{\pr \alpha_4}{\pr t}\frac{\pr(\alpha_2 \alpha_7)}{\pr s}-\frac{\pr \alpha_4}{\pr s}\frac{\pr(\alpha_2 \alpha_7)}{\pr t}= 0,
\end{split}
\end{align}
and from \eqref{coasseq3}, we get 
\begin{align}
\begin{split}
\label{smpl2} 
&\frac{\pr \alpha_1}{\pr s}\left(\frac{\pr\left(\alpha_4^2-\left(\frac{c^2}{(\alpha_4\alpha_6)^2 }+ \alpha_7^2\right)\right)}{\pr t}\right) -\frac{\pr \alpha_1}{\pr t}\left( \frac{\pr\left(\alpha_4^2-\left(\frac{c^2}{(\alpha_4\alpha_6)^2 }+ \alpha_7^2\right)\right)}{\pr s}  \right)\\
    &\quad+\frac{\pr \alpha_2}{\pr s}\frac{\pr(\alpha_4 \alpha_7)}{\pr t}-\frac{\pr \alpha_2}{\pr t}\frac{\pr(\alpha_4 \alpha_7)}{\pr s} = 0.
\end{split}
\end{align}
Therefore, we can choose arbitrary real-valued functions $\alpha_2, \alpha_4, \alpha_7$ and as we know that $\alpha_6$ is determined from $\alpha_2$ and $\alpha_4$, we can solve the two linear PDEs \eqref{smpl1}, \eqref{smpl2}  for $\frac{\pr \alpha_1}{\pr s}$ and $\frac{\pr \alpha_1}{\pr t}$ (which may not always be solvable) to obtain examples of  $\bT^2$-invariant coassociatives. 

\section{Future directions}

We gave an explicit example of a cohomogeneity two $\bT^2$-invariant coassociative and an ansatz to construct such coassociatives using real-valued functions. It remains to be seen if we can come up with more examples using these equations. It will also be interesting to see if these coassociatives lie in some larger family, such as the ones characterized by the pointwise behavior of the second fundamental form (see \cite{danielfoxphd} for details).  It is yet to be determined how many coassociative submanifolds can be constructed this way and what the dimension of such families would be. \\

We can also explore if the methods used in this paper to construct calibrated submanifolds can be employed to find examples of other calibrated submanifolds with symmetries. Furthermore, we can check if these methods can be generalized to obtain examples of a special class of maps known as Smith maps (see \cite{bubbletree} for details). \\

The constructions in this paper are mostly local but it would be enlightening to look at their global properties. We can check if the solutions can be extended globally and if they develop any singularities. In addition, one can try to come up with examples which have a non-trivial global topology.\\

\textbf{Data availability} Data sharing not applicable to this article as no datasets were generated or analysed during the current study.\\

\textbf{Funding} The author declares that no funds, grants, or other support were received during the preparation
of this manuscript.

\section*{Declarations}
\textbf{Conflict of interest} The author declares that he has no financial or non-financial interests to disclose. \\

\newpage
\bibliography{biblo}

\end{document}